\documentclass[french,english]{article} 
\usepackage[utf8]{inputenc}
\usepackage[T1]{fontenc}
\usepackage{lmodern}
\usepackage[a4paper]{geometry}
\usepackage{babel}

\usepackage{fullpage}

\usepackage{amssymb,amsmath,mathtools,amsthm,dsfont}
\usepackage{graphicx}
\usepackage{color}

\usepackage[all]{xy}

\newtheorem{theorem}{Theorem}
\newtheorem{prop-f}[theoreme]{Proposition}
\newtheorem{prop}[theorem]{Proposition}

\newtheorem{corollary}[theorem]{Corollary}

\newtheorem{lemma}[theorem]{Lemma}
\newtheorem{claim}[theorem]{Claim}

\newcommand{\E}{\mathbb{E}}

\newcommand{\N}{\mathbb{N}}

\renewcommand{\P}{\mathbb{P}}
\newcommand{\Q}{\mathbb{Q}}
\newcommand{\R}{\mathbb{R}}

\newcommand{\Z}{\mathbb{Z}}

\newcommand{\cA}{\mathcal{A}}
\newcommand{\cB}{\mathcal{B}}
\newcommand{\cC}{\mathcal{C}}

\newcommand{\cF}{\mathcal{F}}

\newcommand{\cP}{\mathcal{P}}
\newcommand{\cQ}{\mathcal{Q}}
\newcommand{\cR}{\mathcal{R}}
\newcommand{\cS}{\mathcal{S}}

\renewcommand{\1}{\mathds{1}}

\renewcommand{\epsilon}{\varepsilon}
\renewcommand{\phi}{\varphi}

\newcounter{numeroexo}

\newcommand{\card}{\mbox{card}}

\renewcommand{\d}{\text{d}}

\title{\LARGE Continuity of the time constant in a continuous model\\ of first passage percolation\footnote{
Research was partially supported by the ANR project PPPP (ANR-16-CE40-0016) and the Labex MME-DII (ANR 11-LBX-0023-01).}
}
\author{Jean-Baptiste Gou\'er\'e\footnote{Institut Denis-Poisson - UMR CNRS 7013, Universit\'e de Tours, Parc de Grandmont, 37200 Tours, France, {\it jean-baptiste.gouere@lmpt.univ-tours.fr}}  and Marie Th\'eret\footnote{Modal'X, UPL, Universit\'e Paris Nanterre, 92000 Nanterre, France, and FP2M, CNRS FR 2036, {\it marie.theret@parisnanterre.fr}}}
\date{}

\begin{document}

\selectlanguage{english}

\maketitle

\thispagestyle{empty}

\noindent
{\bf Abstract:} {For a given dimension $d\geq 2$ and a finite measure $\nu$ on $(0,+\infty)$, we consider $\xi$ a Poisson point process on $\R^d\times (0,+\infty)$ with intensity measure $dc \otimes \nu$ where $dc$ denotes the Lebesgue measure on $\R^d$. We consider the Boolean model $\Sigma = \cup_{(c,r) \in \xi} B(c,r)$ where $B(c,r)$ denotes the open ball centered at $c$ with radius $r$. For every $x,y \in \R^d$ we define $T(x,y)$ as the minimum time needed to travel from $x$ to $y$ by a traveler that walks at speed $1$ outside $\Sigma$ and at infinite speed inside $\Sigma$. By a standard application of Kingman sub-additive theorem, one easily shows that $T(0,x)$ behaves like $\mu \|x\|$ when $\|x\|$ goes to infinity, where $\mu$ is a constant named the time constant in classical first passage percolation. In this paper we investigate the regularity of $\mu$ as a function of the measure $\nu$ associated with the underlying Boolean model.}
\\


\section{Introduction and main results}

\subsection{The Boolean model}
\label{s:boolean-model}

Fix $d \ge 2$.
Let $\nu$ be a finite measure on $(0,+\infty)$ which is not the null measure.
Let $\xi$ be a Poisson point process on $\R^d \times (0,+\infty)$ with intensity measure $dc \otimes \nu$ 
where $dc$ denotes the Lebesgue measure on $\R^d$.
The Boolean model is the set
\[
\Sigma=\Sigma(d,\nu)=\bigcup_{(c,r) \in \xi} B(c,r)
\]
where $B(c,r)$ denotes the open Euclidean ball centered at $c$ with radius $r$.
We say that $\xi$ and $\Sigma$ are {\em driven} by the measure $\nu$.

Denote by $\lambda$ the total mass of $\nu$, that is $\lambda = \nu[(0,+\infty)]$.
Let $\chi$ the projection of $\xi$ on $\R^d$.
This is a Poisson point process on $\R^d$ with intensity measure $\lambda|\cdot|$.
Almost surely, for all $c \in \chi$ there is a unique $r(c) \in (0,+\infty)$ such that $(c,r(c))$ belongs to $\xi$.
In other words, we can write
\[
\xi = \{(c,r(c)), c \in \chi\}
\]
and thus
\[
\Sigma = \bigcup_{c \in \chi} B(c,r(c)) .
\]
To simplify some notations, we adopt the following convention: $r(c)=0$ if $c$ belongs to $\R^d\setminus\chi$.
Condition to $\chi$, $(r(c))_{c \in \chi}$ is a family of independent random variables with distribution $\lambda^{-1} \nu$.

We refer to the books by Meester and Roy \cite{Meester-Roy-livre} and by Last and Penrose \cite{Last-Penrose-livre} for background on 
Poisson processes and the Boolean model.

When
\begin{equation}\label{e:momentd}
\int_{(0,+\infty)} r^d \nu(\d r) < \infty
\end{equation}
does not hold, $\Sigma=\R^d$ with probability one and the models we are interested in are trivial.
Therefore, in the whole of this work, we assume that all our point processes are driven by measures satisfying \eqref{e:momentd}.

\subsection{Paths}

A path $\pi$ is a finite sequence of distinct points of $\R^d$.
The length of a path $\pi=(x_0,\dots,x_n)$ is the sum of the Euclidean lengths of its segments:
\[
\ell(\pi) = \sum_{i=1}^{n} \|x_i-x_{i-1}\|
\]
where $\|\cdot\|$ denotes the Euclidean norm on $\R^d$.
When $(x_0,\dots,x_n)$ is a sequence of points of $\R^d$ (non necessarily distinct) we define its length by the same formula.
The path is said to be inside $A \subset \R^d$ if all its segments $[x_i,x_{i+1}]$ are included in $A$.
We say that $\pi$ is a path from $a \in \R^d$ to $b \in \R^d$ if its first point is $a$ and its last point is $b$.
We say that $\pi$ is a path from $A \subset \R^d$ to $B \subset \R^d$ if its first point belongs to $A$ and its last points belongs to $B$.
We will occasionally see $\pi$ as a polygonal curve $[0,\ell(\pi)] \to \R^d$ parametrized by arc length.
	
\subsection{First passage percolation in the Boolean model}

\subsubsection{Model} 

A traveler walks on $\R^d$.
He travels at speed infinite inside the Boolean model $\Sigma$ and at speed $1$ outside.
We denote by $T(a,b)$ the time needed to go from $a$ to $b$ along the quickest path.

Here is a formal definition.
For any $x,y\in \R^d$ we define $\tau(x,y)$ as the one-dimensional measure of $[x,y] \cap \Sigma^c$, {\em i.e.}, the sum of the lengths of the segments that constitute $[x,y] \cap \Sigma^c$.
With each path $\pi=(x_0,\dots,x_n)$ we associate a time $\tau(\pi)$ defined by
\[
\tau(\pi)=\sum_{i=1}^n \tau(x_{i-1},x_i).
\]
For any $a,b \in \R^d$ we then set 
\[
T(a,b) = \inf_{\pi} \tau(\pi)
\]
where the infimum is taken along the set of paths from $a$ to $b$\footnote{
See Section \ref{s:mesurabilite} in Appendix for remarks on the measurability of $T$.}.
We define similarly $T(A,B)$ for two subsets
 $A,B \subset \R^d$ as 
\[
T(A,B) =\inf_{a\in A, b\in B} T(a,b).
\]

This model was implicitly introduced in \cite{GM-deijfen} by Régine Marchand and the first author.
It was then explicitly introduced and studied in \cite{Gouere-Theret-17} by the two authors.

\subsubsection{Time constant} 
A standard application of Kingman subadditive ergodic theorem yields the following result.

\begin{theorem}[\cite{Gouere-Theret-17}] \label{t:timeconstant} There exists a constant $\mu(\nu)=\mu(d,\nu) \in [0,1]$ such that:
\[
\lim_{\|x\|\to\infty} \frac{T(0,x)}{\|x\|} = \mu(\nu) \text{ almost surely and in } L^1.
\]
\end{theorem}

The constant $\mu(\nu)$ is called the time constant of the model. We enlighten the fact that the convergence stated in Theorem \ref{t:timeconstant} is uniform with respect to the directions, in particular Theorem \ref{t:timeconstant} implies for instance that $ \lim_{r\rightarrow \infty} T(0,B(0,r)^c) / r = \mu (\nu)$ a.s. Moreover, subadditivity ensures that $\mu (\nu)$ can be defined as an infimum, for instance
\begin{equation}
\label{e:nuinfimum}
\mu (\nu ) = \inf_k \frac{\mathbb{E} [T (0, ke_1)]}{k}
\end{equation}
where $e_1 = (1,0, \dots, 0)$.

\subsubsection{Positivity of the time constant} 
Consider now the condition
\begin{equation}\label{e:greedy}
 \int_{(0,+\infty)} \nu([r,+\infty))^{1/d} dr < \infty.
\end{equation}
We call it the {\em greedy condition}.
It appears in the paper by Martin \cite{Martin-greedy} about greedy lattice paths and animals. 
We refer to \cite{Martin-greedy} for a discussion about \eqref{e:greedy}.
For example, for any $\epsilon>0$,
\begin{equation}\label{e:greedy2}
\int_{(0,+\infty)} r^d \ln_+(r)^{d-1+\epsilon} \nu(dr) < \infty 
\Rightarrow \int_{(0,+\infty)} \nu([r,+\infty))^{1/d} dr < \infty
\Rightarrow \int_{(0,+\infty)} r^d \nu(dr) < \infty.
\end{equation}
In particular, notice that the greedy condition \eqref{e:greedy} implies condition \eqref{e:momentd}.

We say that $\nu$ is strongly subcritical for percolation if 
\begin{equation}
\label{e:nupositif}
\limsup_{r \to \infty} \P[\text{there exists a path inside }\Sigma\text{ from }B(0,r)\text{ to }B(0,2r)^c] = 0.
\end{equation}
The following theorem is the main result of \cite{Gouere-Theret-17}.

\begin{theorem}[\cite{Gouere-Theret-17}] \label{t:gt17}
If \eqref{e:greedy} holds, then $\mu(\nu)$ is positive if and only if $\nu$ is strongly subcritical for percolation.
\end{theorem}
We refer to \cite{Gouere-Theret-17} for more details.
When \eqref{e:momentd} does not hold, $\Sigma=\R^d$ almost surely and therefore $\mu(\nu)=0$.
In the narrow regime where \eqref{e:momentd} holds and \eqref{e:greedy} does not hold, we do not know if $\mu(\nu)=0$ and we have no conjecture.

\subsection{Main results}

\subsubsection{Goal} 

In this paper, we investigate the continuity of $\nu \mapsto \mu(\nu)$.
Here is a weak but simple version of our main result.
Let $(R_n)_n$ be a sequence of positive random variables.
Let $R_\infty$ and $\hat R$ be two positive random variables.
Let $\lambda >0$.
Let $\eta>0$.
Make the following assumptions:
\begin{itemize}
\item For all $n \in \N$, $R_n \le \hat R$.
\item $\E(\hat R^{d+\eta})$ is finite.
\item $\lim_{n \to \infty} \E(|R_n-R_\infty|^{d+\eta}) = 0$.
\end{itemize}
Then
\[
\lim_{n \to \infty} \mu(\lambda \P^{R_n}) = \mu(\lambda \P^R)
\]
where, for any random variable $X$, $\P^X$ denotes the law of $X$.
This result is stated in Corollary \ref{c:mainsuitesuite}.

To state a stronger version of this result, we first need to explain how we couple Boolean model with different driving measures $\nu$.

\subsubsection{Couplings, domination and notations} 
\label{s:couplings}

\paragraph{Coupling.}
Let $\Xi$ be a Poisson point process on $\R^d \times (0,+\infty)$ with intensity measure the Lebesgue measure $\d c \otimes d u$.
We define an admissible map as a map $\cR$ from $(0,+\infty)$ to $[0,+\infty)$ such that
$\cR$ is measurable and $\cR$ vanishes outside a set of finite Lebesgue measure.
For any admissible map we define a Poisson point process $\xi^\cR$ on $\R^d \times (0,+\infty)$ by
\begin{equation}\label{e:xiecriture1}
\xi^\cR = \{(c,\cR(u)), (c,u) \in \Xi \text{ such that } \cR(u)>0\}
\end{equation}
and a measure $\nu^\cR$ on $(0,+\infty)$ by
\[
\nu^\cR(A) = \int_{(0,+\infty)} \1_A(\cR(u)) \d u.
\]
As $\cR$ vanishes outside a set of finite Lebesgue measure, $\nu^\cR$ is a finite measure.
The process $\xi^\cR$ is a Poisson point process on $\R^d \times (0,+\infty)$ with intensity measure $\d c \otimes \nu^\cR$.
This construction of the process will enable us to couple such processes with different $\cR$ in a natural way.

\paragraph{Notations.}
We define $r^\cR:\R^d \to \R_+$ as follows. 
Let $c \in \R^d$.
If there exists $u \in (0,+\infty)$ such that $(c,u) \in \Xi$, then such a $u$ is unique and therefore it makes sense to set $r^\cR(c)=\cR(u)$.
Otherwise, we set $r^\cR(c)=0$.
We thus can rewrite  \eqref{e:xiecriture1} as follows:
\begin{equation}\label{e:xiecriture2}
\xi^\cR = \{(c,r^\cR(c)), c \in \R^d : r^\cR(c)>0\}.
\end{equation}

\paragraph{Representation of a measure.} For any finite measure $\nu$ on $(0,+\infty)$, we define an admissible map $\cR_\nu$ by 
\begin{equation} \label{e:skorokhod}
\cR_\nu(u)=\sup \{r > 0 : \nu([r,+\infty) ) \ge u\}
\end{equation}
with the convention $\cR_\nu(u)=0$ if the set is empty.
This admissible map is such that $\nu^{\cR_\nu}=\nu$.
This is a variation on the notion of generalized inverse distribution function\footnote{Here is a proof
of the fact that $\cR_\nu$ is an admissible map such that $\nu^{\cR_\nu}=\nu$.
The map $\cR_\nu$ defined by \eqref{e:skorokhod} is non-increasing and therefore measurable.
As the map moreover takes values in $[0,+\infty)$ and vanishes for $u$ larger than the total mass of $\nu$, $\cR_\nu$ is an admissible map.
The fact that $\nu^{\cR_\nu}=\nu$ is a consequence of the fact that, for any $r>0$, we have
\[
\{u>0 : \cR_\nu(u) \ge r\} = \{u>0 : u \le \nu([r,+\infty))\}
\]
and thus $\nu^{\cR_\nu}([r,+\infty)) =  \nu([r,+\infty))$. Let us prove the above equality.
If $u \le \nu([r,+\infty))$, then $\cR_\nu(u) \ge r$ by definition of $\cR_\nu$.
If $u > \nu([r,+\infty))$, then there exists $s \in (0,r)$ such that $u > \nu([s,+\infty))$.
In this case,  by definition of $\cR_\nu$, we have $\cR_\nu(u) \le s < r$. 
This ends the proof.
}.

\paragraph{Domination and coupling.} Let $\nu_1$ and $\nu_2$ be two finite measures on $(0,+\infty)$. 
We say that $\nu_1$ is dominated by $\nu_2$, which we denote by $\nu_1 \prec \nu_2$, if, for all $r > 0$,
\[
\nu_1([r,+\infty)) \le \nu_2([r,+\infty)).
\]
If $\nu_1 \prec \nu_2$, then $\cR_{\nu_1} \le \cR_{\nu_2}$.
The converse is true.
If $\cR_1$ and $\cR_2$ are two admissible maps such that $\cR_1 \le \cR_2$, then $\nu^{\cR_1} \le \nu^{\cR_2}$. 

Let $\nu_1$ and $\nu_2$ be two finite measures on $(0,+\infty)$ such that $\nu_1 \prec \nu_2$.
Using the coupling described above, we can define $\xi_1:=\xi^{\cR_{\nu_1}}$ and  $\xi_2:=\xi^{\cR_{\nu_2}}$.
These are two Poisson point processes driven by $\nu_1$ and $\nu_2$.
For short, we write $r_1 := r^{\cR_{\nu_1}}$ and  $r_2 := r^{\cR_{\nu_2}}$.
We can write
\[
\xi_1 = \{(c,r_1(c)), c \in \R^d : r_1(c) > 0\} \text{ and } \xi_2 = \{(c,r_2(c)), c \in \R^d : r_2(c) > 0\}.
\]
As $\cR_{\nu_1} \le \cR_{\nu_2}$, we have, for all $c \in \R^d$, $r_1(c) \le r_2(c)$. 
This is a consequence of the link between $r_i$ and $\cR_{\nu_i}$, see above.

\subsubsection{Main results} 

Here is our main result. 

\begin{theorem} \label{t:main} Let $\cR_1, \cR_2$ and $\hat\cR$ be three admissible maps.
Assume
\begin{enumerate}
\item $\cR_1 \le \hat\cR$ and $\cR_2 \le \hat\cR$.
\item $\int_0^\infty \nu^{\hat\cR}([r,+\infty))^{1/d} \d r < \infty$.
\item $\mu(\nu^{\hat\cR})>0$.
\end{enumerate}
There exists $C=C(d,\hat\cR)$ such that 
\[
\left|\mu(\nu^{\cR_1})-\mu(\nu^{\cR_2})\right| \le C \int_0^\infty \nu^{|\cR_1-\cR_2|}([r,+\infty))^{1/d} \d r.
\]
\end{theorem}

The following simple consequence is weaker.
We estimate that it is worth stating because its statement is simpler as it involves more familiar quantities.

\begin{corollary} \label{c:main} Let $\cR_1, \cR_2$ and $\hat\cR$ be three admissible maps.
Let $\eta>0$.
Assume
\begin{enumerate}
\item $\cR_1 \le \hat\cR$ and $\cR_2 \le \hat\cR$.
\item $ \int_0^\infty \hat\cR(u)^{d+\eta} \d u< \infty$.
\item $\mu(\nu^{\hat\cR})>0$.
\end{enumerate}
There exists $C=C(d,\hat\cR,\eta)$ such that 
\[
\left|\mu(\nu^{\cR_1})-\mu(\nu^{\cR_2})\right| \le C \left(\int_0^\infty |\cR_1(u)-\cR_2(u)|^{d+\eta} \d u\right)^{1/(d+\eta)}.
\]
\end{corollary}

\paragraph{Plan of the proof of Theorem \ref{t:main}.}
Consider two admissible maps $\cR_1, \cR_2$, with $\cR_1 \leq \cR_2$ for instance, and denote for short by $T_i (s)$ the passage time from $0$ to $B(0,s)^c$ ($s>0$) associated with the point process $\xi^{\cR_i}$, and write $r_i =r^{\cR_i}$. By coupling we get $T_2 (s) \leq T_1 (s)$, and looking at the difference of passage time for the two processes along a nicely behaved geodesic $\pi$ for $T_2(s)$, we obtain
\[
T_1 (s) \leq T_2 (s) + 2 \sum_{x\in \pi} (r_2 (x) - r_1 (x)),
\]
thus
\[
\frac{T_1 (s)}{s} \leq \frac{T_2 (s)}{s} + 2 \frac{\ell (\pi)}{s} \frac{\sum_{x\in \pi} (r_2 (x) - r_1 (x))}{\ell (\pi)},
\]
where $\ell(\pi)$ is the Euclidean length of $\pi$, defined as the sum of the Euclidean length of its segments.
At this stage, we need two tools:
\begin{itemize}
\item[$(i)$] a control on $\ell (\pi) /s$: it is given by Theorem \ref{t:controle-l}; this is the main part of the proof, and we obtain roughly speaking the asymptotic upper bound $\ell (\pi) /s \leq C$ for a constant $C$ and large enough $s$. The plan of the proof of Theorem \ref{t:controle-l} is given just after its statement.
\item[$(ii)$] a control on $\sum_{x\in \pi} (r_2 (x) - r_1 (x))/\ell (\pi)$: it is given by known results on the so called greedy paths (see Corollary \ref{c:greedy}), and we obtain the asymptotic upper bound $\sum_{x\in \pi} (r_2 (x) - r_1 (x))/\ell (\pi) \leq C' \int_0^\infty \nu^{\cR_2 - \cR_1} ([r,+\infty))^{1/d} dr$ for a constant $C'$ and large enough $s$.
\end{itemize}
Letting $s$ go to infinity, this gives the desired control on $|\mu(\nu^{\cR_1}) - \mu(\nu^{\cR_2})|$. The hypotheses required in Theorem \ref{t:main} find their origins in the controls $(i)$ and $(ii)$ described above. These controls must be uniform in the maps $\cR_i$, thus a domination by a map $\hat \cR$ is required. Step $(i)$ is a control on the length of a geodesic: it is not surprising that it requires to consider models in the subcritical regime of first passage percolation, i.e., such that $\mu (\hat \cR) >0$ (see the remark below the statement of Theorem  \ref{t:controle-l} for more details). Step $(ii)$ makes use of results on greedy paths, that require the greedy condition $\int_0^\infty \nu^{\hat\cR}([r,+\infty))^{1/d} \d r < \infty$ to hold.

The following result is also a corollary of Theorem \ref{t:main}.

\begin{theorem} \label{t:mainsuite} Let $\cR_\infty$ and $\hat\cR$ be two admissible maps.
Let $(\cR_n)_n$ be a sequence of admissible maps.
Assume 
\begin{enumerate}
\item For all $n \in \N \cup \{\infty\}$, $\cR_n \le \hat\cR$.
\item $\int_0^\infty \nu^{\hat\cR}([r,+\infty))^{1/d} \d r < \infty$.
\item $\cR_n$ converges almost everywhere (with respect to the Lebesgue measure) to $\cR_\infty$.
\end{enumerate}
Then 
\[
\lim_{n \to \infty} \mu(\nu^{\cR_n}) =  \mu(\nu^{\cR_\infty}).
\]
\end{theorem}

\paragraph{Plan of the proof of Theorem \ref{t:mainsuite}.}
It can be divided into two parts that are, roughly speaking, the upper bound $\limsup \mu (\nu^{\cR_n}) \leq \mu(\nu^{\cR_\infty}) $, and the lower bound $\liminf \mu (\nu^{\cR_n}) \geq \mu(\nu^{\cR_\infty})$. The proof of the upper bound is the easy part, and relies on the definition of the time constant as an infimum. The proof of the lower bound is the most delicate part. It is trivial when $\mu (\nu ^{\cR _\infty}) =0$. Otherwise, it is a consequence of Theorem \ref{t:main}.

We state as a corollary a special case of Theorem \ref{t:mainsuite}.
We denote by $\P_X$ the distribution of a random variable $X$.

\begin{corollary} \label{c:mainsuitesuite}
Let $R_\infty$ and $\hat R$ be two positive random variables.
Let $(R_n)_n$ be a sequence of positive random variables.
Let $\lambda_\infty , (\lambda_n)_n$ be positive real numbers.
Assume 
\begin{enumerate}
\item For all $n \in \N \cup \{\infty\}$, $\P^{R_n}$ is dominated by $\P^{\hat R}$.
\item $\int_0^\infty \big(\P[\hat R \ge r]\big)^{1/d} \d r < \infty$.
\item $P^{R_n}$ converges weakly to $ \P^{R_\infty}$.
\item $\lim_{n\rightarrow \infty} \lambda_n = \lambda_\infty$.
\end{enumerate}
Then 
\[
\lim_{n \to \infty} \mu\big(\lambda_n \P^{R_n}\big) =   \mu\big(\lambda_\infty \P^{R_\infty}\big).
\]
\end{corollary}

\subsection{Organization of the paper}
\label{s:plan}

In Section \ref{s:skeleton} we gather vocabulary and known results concerning greedy paths (see step $(ii)$ in the plan of the proof of Theorem \ref{t:main} above). In Section \ref{s:main-controle-longueur} we obtain some uniform controls on the length of a geodesic (see step $(i)$). This is the main part of the proof and we believe that some of the results are of independent interest. In Section \ref{s:preuve:main} we prove Theorem \ref{t:main}, Corollary \ref{c:main}, Theorem \ref{t:mainsuite} and Corollary \ref{c:mainsuitesuite}.

\section{Control on the length of a good geodesic}
\label{s:controle-longueur}

\subsection{Regular paths, good paths, geodesics}
\label{s:skeleton}

\paragraph{Framework and conventions.} 
Let $\nu$ be a finite measure on $(0,+\infty)$.
Let $\xi$ be a Poisson point process driven by $\nu$ (see Section \ref{s:boolean-model}).
As usual, we write
\[
\xi = \{(c,r(c)), c \in \chi\}
\]
where $\chi$ is the projection of $\xi$ on $\R^d$.

We will define several objects depending on $\xi$.
In some parts of the paper, we will deal simultaneously with several such point processes.
In these cases, we will write for example 
"$\xi$-good" instead of "good", "$T(a,b;\xi)$" instead of "$T(a,b)$" and so on.
In some part of the proofs, we will shorten such notations and write for example 
"$T_n^+(a,b)$" instead "$T(a,b ; \xi_n^+)$", "$\hat T(a,b)$" instead "$T(a,b ; \hat\xi)$" and so on.
Each time we abbreviate the notations in this way, we will explicitly mention it.

\paragraph{Regular paths.}
We say that a path $\pi=(x_0,\dots,x_n)$ is regular if, for all $i \in \{1,\dots,n-1\}$, $x_i$ belongs to $\chi$.
Note that we do not require anything about $x_0$ and $x_n$.

\paragraph{Local travel times.} Let $\pi=(x_0,\dots,x_n)$ be a path.
For any $i \in \{1,\dots,n-1\}$, we define
\begin{equation} \label{e:Z}
Z_i^-(\pi) = [x_{i-1},x_i]  \cap B(x_i,r(x_i)) \text{ and } Z_i^+ = [x_i,x_{i+1}] \cap B(x_i,r(x_i)).
\end{equation}
(see Figure \ref{f:fig1}).
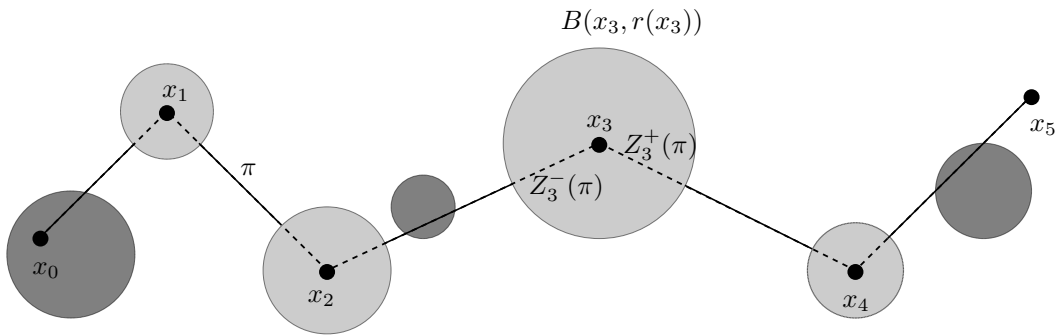
\begin{figure}[!ht]
\centering
\def\svgwidth{0.9\textwidth}
\begingroup%
  \makeatletter%
  \providecommand\color[2][]{%
    \errmessage{(Inkscape) Color is used for the text in Inkscape, but the package 'color.sty' is not loaded}%
    \renewcommand\color[2][]{}%
  }%
  \providecommand\transparent[1]{%
    \errmessage{(Inkscape) Transparency is used (non-zero) for the text in Inkscape, but the package 'transparent.sty' is not loaded}%
    \renewcommand\transparent[1]{}%
  }%
  \providecommand\rotatebox[2]{#2}%
  \newcommand*\fsize{\dimexpr\f@size pt\relax}%
  \newcommand*\lineheight[1]{\fontsize{\fsize}{#1\fsize}\selectfont}%
  \ifx\svgwidth\undefined%
    \setlength{\unitlength}{510.07523602bp}%
    \ifx\svgscale\undefined%
      \relax%
    \else%
      \setlength{\unitlength}{\unitlength * \real{\svgscale}}%
    \fi%
  \else%
    \setlength{\unitlength}{\svgwidth}%
  \fi%
  \global\let\svgwidth\undefined%
  \global\let\svgscale\undefined%
  \makeatother%
  \begin{picture}(1,0.29767566)%
    \lineheight{1}%
    \setlength\tabcolsep{0pt}%
    \put(0,0){\includegraphics[width=\unitlength,page=1]{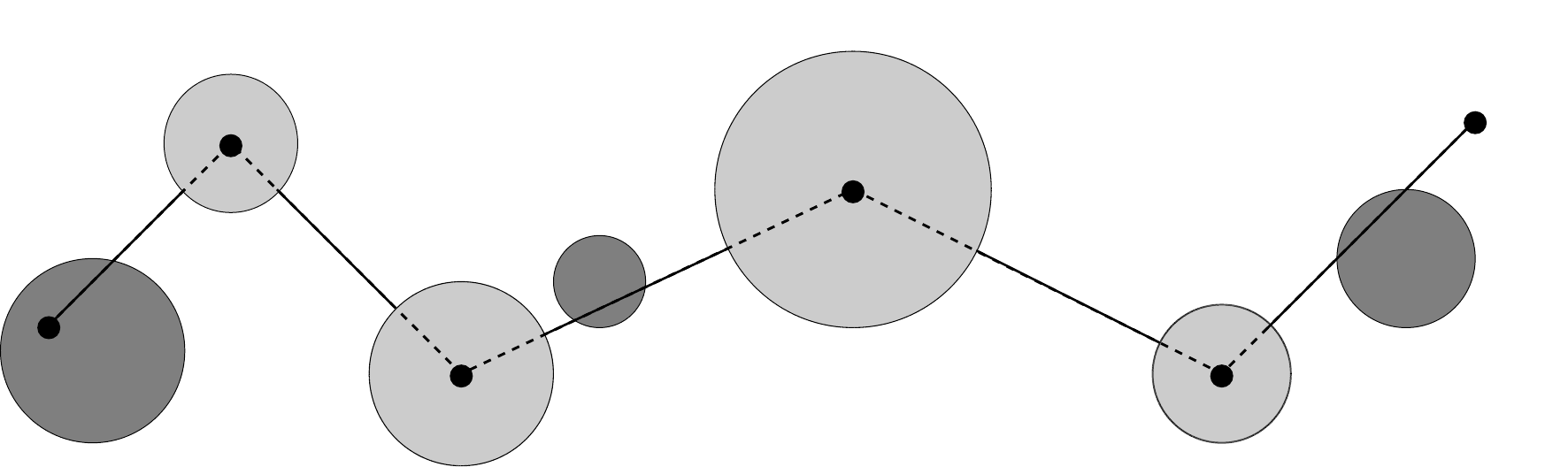}}%
    \put(0.21495187,0.14731486){\color[rgb]{0,0,0}\makebox(0,0)[lt]{\lineheight{1.25}\smash{\begin{tabular}[t]{l}$\pi$\end{tabular}}}}%
    \put(0.53255208,0.191426){\color[rgb]{0,0,0}\makebox(0,0)[lt]{\lineheight{1.25}\smash{\begin{tabular}[t]{l}$x_3$\end{tabular}}}}%
    \put(0.27670748,0.02968515){\color[rgb]{0,0,0}\makebox(0,0)[lt]{\lineheight{1.25}\smash{\begin{tabular}[t]{l}$x_2$\end{tabular}}}}%
    \put(0.14290369,0.21936305){\color[rgb]{0,0,0}\makebox(0,0)[lt]{\lineheight{1.25}\smash{\begin{tabular}[t]{l}$x_1$\end{tabular}}}}%
    \put(0.93984493,0.18848526){\color[rgb]{0,0,0}\makebox(0,0)[lt]{\lineheight{1.25}\smash{\begin{tabular}[t]{l}$x_5$\end{tabular}}}}%
    \put(0.76781147,0.02527405){\color[rgb]{0,0,0}\makebox(0,0)[lt]{\lineheight{1.25}\smash{\begin{tabular}[t]{l}$x_4$\end{tabular}}}}%
    \put(0.0238036,0.05615184){\color[rgb]{0,0,0}\makebox(0,0)[lt]{\lineheight{1.25}\smash{\begin{tabular}[t]{l}$x_0$\end{tabular}}}}%
    \put(0.56637064,0.16790006){\color[rgb]{0,0,0}\makebox(0,0)[lt]{\lineheight{1.25}\smash{\begin{tabular}[t]{l}$Z^{+}_{3}(\pi)$\end{tabular}}}}%
    \put(0.47961871,0.12967041){\color[rgb]{0,0,0}\makebox(0,0)[lt]{\lineheight{1.25}\smash{\begin{tabular}[t]{l}$Z^{-}_{3}(\pi)$\end{tabular}}}}%
    \put(0.50902608,0.28258903){\color[rgb]{0,0,0}\makebox(0,0)[lt]{\lineheight{1.25}\smash{\begin{tabular}[t]{l}$B(x_3, r(x_3))$\end{tabular}}}}%
  \end{picture}%
\endgroup%
\caption{The path $\pi=(x_0,\dots, x_5)$ ; the set $Z(\pi)$ is represented in dashed line ; the local travel time $\widetilde{\tau} (\pi)$ of $\pi$ is equal to the Euclidean length of the parts of $\pi$ that are in solid line, including parts of $\pi$ that are in $\Sigma$ but not in $Z(\pi)$ (see the intersection of $\pi$ with the balls of $\Sigma$ that are represented in dark grey).} 
\label{f:fig1}
\end{figure}
We also set
\[
Z_0^+(\pi) = Z_n^-(\pi) = \emptyset.
\]
We define a new travel time by
\[
\widetilde\tau(\pi) = \sum_{i=1}^n \ell\left([x_{i-1},x_i] \setminus \left( Z^+_{i-1}(\pi) \cup Z^-_i(\pi) \right)\right).
\]
In words, this is the time needed for a walker to travel along $\overline{\pi}$ if it travels at speed $1$ outside 
\[
Z(\pi)=\bigcup_{i=1}^{n-1} \left(Z_i^-(\pi) \cup Z_i^+(\pi)\right)
\]
and at speed infinite inside $Z(\pi)$.

\paragraph{Good paths.}
Let $\pi=(x_0,\dots,x_n)$ be a path.
We say that $\pi$ is a good path if the following conditions are fulfilled:
\begin{enumerate}
 \item For all $i \in \{1,\dots,n-1\}$, $x_i$ belongs to $\chi$. In other words, $\pi$ is regular.
 \item For all $i \in \{1,\dots,n\}$, $[x_{i-1},x_i] \cap\left( Z^+_{i-1}(\pi) \cup Z^-_i(\pi) \right) = [x_{i-1},x_i] \cap \Sigma$.
\end{enumerate}

\begin{lemma} \label{l:good}
 If $\pi=(x_0,\dots,x_n)$ is a path, then
 $
 \widetilde\tau(\pi) \ge \tau(\pi).
 $
 If moreover $\pi$ is $\xi$-good, then
 $
 \widetilde\tau(\pi) = \tau(\pi).
 $
\end{lemma}
\begin{proof}
 The inclusion
 $
 Z(\pi) \subset \Sigma
 $
 always holds by definition of $Z(\pi)$.
 Therefore 
 \[
 \overline{\pi} \cap Z(\pi) \subset \overline\pi \cap \Sigma.
 \]
 When $\pi$ is a good path, the previous inclusion is an equality by definition of good paths.
 The proof then follows by definition of $\tau$ and $\widetilde\tau$. 
\end{proof}

\paragraph{Geodesics and good geodesics.} Let $r>0$.
For any $c \in \R^d$, denote by $S(c,r)$ the sphere of radius $r$ centered at $c$.
 A geodesic from $0$ to $S(0,r)$ is a path $\pi$ from $0$ to $S(0,r)$ such that $\tau(\pi)=T(0,S(0,r))$.
A good geodesic is a geodesic which is, moreover, a good path.

\begin{lemma} \label{l:good-geodesic}
Let $\nu$ be a finite measure on $(0,+\infty)$ satisfying the moment condition \eqref{e:momentd}.
For any $r>0$, with probability one, there exists a good geodesic from $0$ to $S(0,r)$.
\end{lemma}

To prove Lemma \ref{l:good-geodesic}, we need the following intermediate lemma.

\begin{lemma} \label{l:good-geodesic-cas-fini}
Let $A$ and $B$ be two compact subsets of $\R^d$.
Let $\xi^f$ be a finite subset of $\R^d \times (0,+\infty)$.
There exists a path $\pi$ from $A$ to $B$ such that $\widetilde\tau(\pi ; \xi^f)= T(A,B ; \xi^f)$.
\end{lemma}

\begin{proof}
Let $\Sigma^f$ denote the union of all the balls $B(c,r)$ for $(c,r) \in \xi^f$.
Let $V$ be the set of connected components of $\Sigma^f$.
Define $V'$ as the union of $V$ and $\{A,B\}$.
The set $V'$ is finite.
We consider the complete graph whose vertices set is $V'$.
The length of an edge between any $C,C'$ in $V'$ is defined as the Euclidean distance $d(C,C')$ between $C$ and $C'$.
We consider the natural associated geodesic distance $d_G$ on the graph.
Let us prove the equality
\begin{equation}\label{e:geodesique-discretisation}
T(A,B; \xi^f) = d_G(A,B)
\end{equation}
and the existence of a path $\pi$ such that $\widetilde\tau(\pi ; \xi^f)= T(A,B ; \xi^f)$.

We first prove the inequality $T(A,B; \xi^f) \ge d_G(A,B)$.
Let $\pi$ be a path from $A$ to $B$.
We see $\pi$ as a curve $[0,\ell(\pi)] \to \R^d$ parametrized by arc-length.
Let $(C(1),\dots,C(n-1))$ be the finite sequence\footnote{The definition makes sense because of the following facts.
\begin{itemize}
 \item The set $V$ is finite.
 \item The sets $\pi^{-1}(C), C \in V$ are disjoint.
 \item For each $C \in V$, the set $\pi^{-1}(C)$ is the union of finite number of intervals.
\end{itemize}}
of elements of $V$ successively visited by $\pi$. 
Set $C(0)=A$ and $C(n)=B$.
Thus, $(C(0),\dots,C(n))$ is a sequence of elements of $V'$.
For any $i \in \{0,\dots,n-1\}$, some part (possibly empty) of $\pi$ goes from $C(i)$ to $C(i+1)$ without touching $\Sigma^f$.
The travel time of this part of $\pi$ is at least $d(C(i),C(i+1))$ and thus
\[
\tau(\pi ; \xi^f) \ge \sum_{i=0}^{n-1} d\big(C(i),C(i+1)\big) \ge d_G(A,B).
\]
Therefore 
\begin{equation}\label{e:gaman}
T(A,B; \xi^f) \ge d_G(A,B).
\end{equation}

We now build a path $\pi$ such that
\begin{equation}\label{e:manga}
\widetilde\tau(\pi;\xi^f) \le d_G(A,B).
\end{equation}
Let $(C_0=A,C(1),\dots,C(n)=B)$ be a sequence of distinct vertices of $V'$ such that
\[
d_G(A,B) = \sum_{i=0}^{n-1} d\big(C(i),C(i+1)\big).
\]
If $n=1$ we set $\pi=(a,b)$ where $a\in A$ and $b \in B$ are such that $\|b-a\| = d(A,B)$.
The conclusion is then straightforward.
Henceforth we assume $n \ge 2$.
The path $\pi$ is obtained by a natural concatenation of paths of the following kinds.
We refer to the figure \ref{f:fig2} for an example.
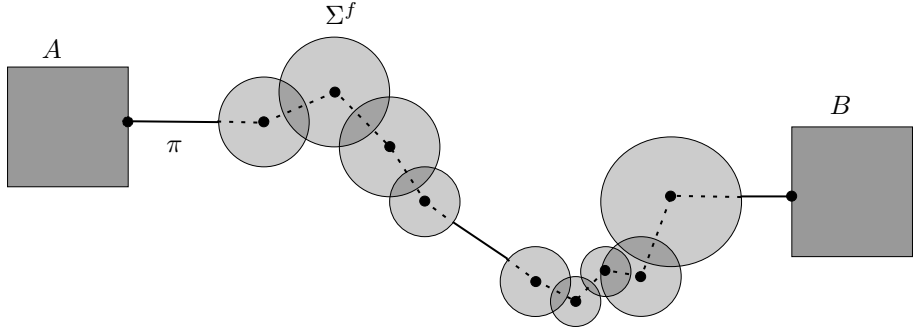
\begin{figure}[!ht]
\centering
\begingroup%
  \makeatletter%
  \providecommand\color[2][]{%
    \errmessage{(Inkscape) Color is used for the text in Inkscape, but the package 'color.sty' is not loaded}%
    \renewcommand\color[2][]{}%
  }%
  \providecommand\transparent[1]{%
    \errmessage{(Inkscape) Transparency is used (non-zero) for the text in Inkscape, but the package 'transparent.sty' is not loaded}%
    \renewcommand\transparent[1]{}%
  }%
  \providecommand\rotatebox[2]{#2}%
  \newcommand*\fsize{\dimexpr\f@size pt\relax}%
  \newcommand*\lineheight[1]{\fontsize{\fsize}{#1\fsize}\selectfont}%
  \ifx\svgwidth\undefined%
    \setlength{\unitlength}{337.78348079bp}%
    \ifx\svgscale\undefined%
      \relax%
    \else%
      \setlength{\unitlength}{\unitlength * \real{\svgscale}}%
    \fi%
  \else%
    \setlength{\unitlength}{\svgwidth}%
  \fi%
  \global\let\svgwidth\undefined%
  \global\let\svgscale\undefined%
  \makeatother%
  \begin{picture}(1,0.3595855)%
    \lineheight{1}%
    \setlength\tabcolsep{0pt}%
    \put(0,0){\includegraphics[width=\unitlength,page=1]{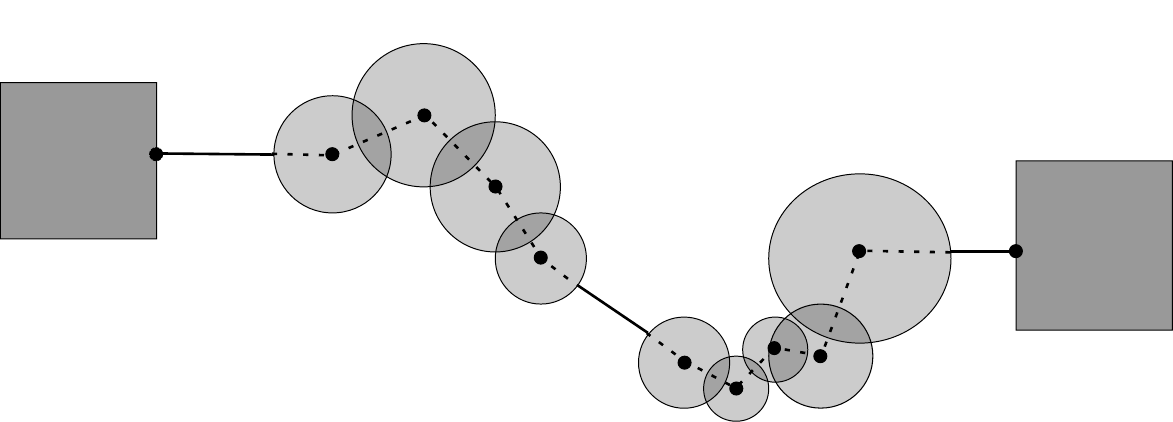}}%
    \put(0.03816567,0.29905762){\color[rgb]{0,0,0}\makebox(0,0)[lt]{\lineheight{1.25}\smash{\begin{tabular}[t]{l}$A$\end{tabular}}}}%
    \put(0.35123604,0.33680369){\color[rgb]{0,0,0}\makebox(0,0)[lt]{\lineheight{1.25}\smash{\begin{tabular}[t]{l}$\Sigma^f$\end{tabular}}}}%
    \put(0.90854577,0.23688762){\color[rgb]{0,0,0}\makebox(0,0)[lt]{\lineheight{1.25}\smash{\begin{tabular}[t]{l}$B$\end{tabular}}}}%
    \put(0.17582781,0.19470082){\color[rgb]{0,0,0}\makebox(0,0)[lt]{\lineheight{1.25}\smash{\begin{tabular}[t]{l}$\pi$\end{tabular}}}}%
  \end{picture}%
\endgroup%
\caption{Construction of the path $\pi$ from $A$ to $B$ - the travel time $\widetilde{\tau} (\pi, \xi^f)$ of $\pi$ is equal to the Euclidean length of the parts of $\pi$ that are in solid line.} 
\label{f:fig2}
\end{figure}
\begin{enumerate}
\item Let $C \in V$. This is a connected component of $\Sigma^f$.
We can write $C = \cup_{(c,r) \in \xi^C} B(c,r)$ where $\xi^C$ is a subset $\xi^f$.
For any distinct $(c,r), (c',r') \in \xi^C$ there exists a sequence $(c,r)=(c_0,r_0), \dots, (c_k,r_k)=(c',r')$ of elements of $\xi^C$ such that
$\widetilde\tau\big((c_0,\dots,c_k) , \xi^f\big)=0$.
Indeed, it suffices to consider a sequence in which each ball $B(c_i,r_i)$ touches the ball $B(c_{i+1},r_{i+1})$.
\item Let $C$ and $C'$ be two distinct elements of $V$.
As before, write $C = \cup_{(c,r) \in \xi^C} B(c,r)$ and $C' = \cup_{(c,r) \in \xi^{C'}} B(c,r)$ where $\xi^C, \xi^{C'} \subset \xi^f$.
There exists $(c,r) \in \xi^C$ and $(c',r') \in \xi^{C'}$ such that
\begin{equation}\label{e:ineq1}
\widetilde\tau\big((c,c') ; \xi^f\big) \le d(C,C').
\end{equation}
Indeed, there exists $x \in \overline{C}$ and $x' \in \overline{C'}$ such that $\|x-x'\| = d(C,C')$.
Then, there exists $(c,r) \in \xi^C$ and $(c',r') \in \xi^{C'}$ such that $x \in S(c,r)$ and $x' \in S(c',r')$.
Note that $x$ and $x'$ belong to the line segment $[c,c']$.
Actually, $[c,c'] \setminus \big( B(c,r) \cup B(c',r')\big) = [x,x']$ and therefore \eqref{e:ineq1} holds.
\item Let $C \in V$ and $K$ be either $A$ or $B$. 
As before, write $C = \cup_{(c,r) \in \xi^C} B(c,r)$ where $\xi^C \subset \xi^f$.
\begin{enumerate}
\item If $C \cap K$ is not empty, consider some $z \in C \cap K$. 
There exists $(c,r) \in \xi^C$ such that $z \in B(c,r)$.
If $c=z$ consider the path $\pi'=(c)$.
Otherwise consider the path $\pi'=(z,c)$.
In any case,
\[
\widetilde\tau\big(\pi';\xi^f\big) = 0 = d(K,C).
\]
\item If $C \cap K$ is empty there exists $(c,r) \in \xi^C$ and $z \in K$ such that 
\begin{equation}\label{e:ineq2}
\widetilde\tau\big((z,c);\xi^f\big)\le d(K,C).
\end{equation}
Let us prove it. 
There exists $z \in K$ and $x \in \overline{C}$ such that $d(K,C)=\|x-z\|$.
Then, as above, there exists $(c,r) \in \xi^C$ such that $x \in S(c,r)$ and $[z,c] \setminus B(c,r) = [z,x]$ 
(here we need $C \cap K = \emptyset$) and thus \eqref{e:ineq2}.
\end{enumerate}
\end{enumerate}
Concatenating paths, we get a path $\pi$ satisfying \eqref{e:manga}.

From \eqref{e:gaman}, \eqref{e:manga} and Lemma \ref{l:good} we get
\[
d_G(A,B) \le T(A,B ; \xi^f) \le \tau(\pi ; \xi^f) \le \widetilde\tau(\pi ; \xi^f) \le d_G(A,B).
\]
Therefore all inequalities are equalities and in particular $\widetilde\tau(\pi ; \xi^f) = T(A,B ; \xi^f)$.
\end{proof}

We can now prove the existence of a good geodesic from $0$ to $S(0,r)$ almost surely.

\begin{proof}[Proof of Lemma \ref{l:good-geodesic}] By definition,
\[
T(0,S(0,r)) = \inf_\pi \tau(\pi)
\]
where $\pi$ is the set of paths from $0$ to $S(0,r)$.
We can furthermore assume that the paths are entirely inside $\overline{B(0,r)}$\footnote{Let indeed $\pi$ be a path from $0$ to $S(0,r)$.
We see it as a parametrized curve.
We stop the path at its first intersection with $S(0,r)$.
We thus get a path $\pi'$ from $0$ to $S(0,r)$ which is entirely inside $\overline{B(0,r)}$ and such that $\tau(\pi') \le \tau(\pi)$.}.
But for any such path $\pi$, the travel time $\tau(\pi)$ only depends on random balls which touches $\overline{B(0,r)}$.
Thus $T(0,S(0,r) ; \xi) = T(0,S(0,r) ; \xi^f)$ where
\[
\xi^f = \{(c,s) \in \xi : B(c,s) \cap \overline{B(0,r)} \neq \emptyset\}.
\]
But as $\nu$ satisfies \eqref{e:momentd} the set $\xi^f$ is almost surely finite.
Therefore, by Lemma \ref{l:good-geodesic-cas-fini}, there exists a $\xi^f$-good geodesic $\pi^f$ from $0$ to $S(0,r)$.
As a consequence,
\begin{align*}
T(0,S(0,r) ; \xi) 
& = T(0,S(0,r) ; \xi^f) \\
& = \widetilde\tau(\pi^f ; \xi^f)  \text{ as } \pi^f \text{ is a } \xi^f \text{-good geodesic}\\
& \ge \widetilde\tau(\pi^f ; \xi) \text{ as } \xi^f \subset \xi \\
& \ge \tau(\pi^f ; \xi) \text{ by Lemma \ref{l:good}} \\
& \ge T(0,S(0,r) ; \xi) \text{ by definition of } T(0,S(0,r) ; \xi).
\end{align*}
Therefore all the inequalities are equalities.
From $\tau(\pi^f ; \xi) = T(0,S(0,r) ; \xi)$ we deduce that $\pi^f$ is a geodesic from $0$ to $S(0,r)$. The fact that $\pi^f$ is almost surely a $\xi$-good path is a consequence of the equality $\widetilde\tau(\pi^f ; \xi) = \tau(\pi^f ; \xi)$, and can be proved as follows.
On a probability one event, the random balls $B(c,s)$ with $(c,s)\in \xi^f$ are pairwise non tangent, and none of them is neither tangent to $S(0,r)$. 
 We work on this full probability event.
 Write $\pi^f=(x_0,\dots,x_n)$. Since $\pi^f$ is $\xi^f$-regular it is also $\xi$-regular. Suppose that $\pi^f$ is not $\xi$-good.
Then there exists $i \in \{1,\dots,n\}$ such that $[x_{i-1},x_i] \cap \left( Z^+_{i-1}(\pi) \cup Z^-_i(\pi) \right) \neq [x_{i-1},x_i] \cap \Sigma$.
As the first set is included in the second set, this implies that $\Sigma \cap I$ is non empty where
 $I = [x_{i-1},x_i] \setminus \left( Z^+_{i-1}(\pi) \cup Z^-_i(\pi) \right)$.
Therefore $I$ is non empty and, on our full probability event, this implies that $I$ is an interval of positive length.
As $\Sigma$ is open this implies that the length of $\Sigma \cap I$ is positive.
Therefore, $\widetilde\tau(\pi^f) > \tau(\pi^f)$, which is a contradiction. This concludes the proof.
\end{proof}

Another useful consequence of Lemma \ref{l:good-geodesic-cas-fini} is the following lemma.

\begin{lemma} \label{l:good-geodesic-pi}
Let $\nu$ be a finite measure on $(0,+\infty)$ satisfying the moment condition \eqref{e:momentd}.
Let $a, b \in \R^d$.
Let $\pi$ be a path from $a$ to $b$.
There exists a path $\widetilde\pi$ from $a$ to $b$ such that $\widetilde\tau(\widetilde\pi) \le \tau(\pi)$.
\end{lemma}

\begin{proof} We see $\pi$ as a curve parametrized by arc-length $\pi : [0,\ell(\pi)] \to \R^d$ and denote by $\overline{\pi}$ its image in $\R^d$.
This is a compact set.
The travel time $\tau(\pi)$ only depends on random balls which touches $\overline\pi$.
Thus $\tau(\pi ; \xi) = \tau(\pi ; \xi^f)$ where
\[
\xi^f = \{(c,r) \in \xi : B(c,r) \cap \overline\pi \neq \emptyset\}.
\]
But as $\nu$ satisfies \eqref{e:momentd} the set $\xi^f$ is almost surely finite.
Therefore, by Lemma \ref{l:good-geodesic-cas-fini}, there exists a $\xi^f$-good geodesic $\widetilde\pi$ from $a$ to $b$.
As a consequence,
\begin{align*}
\tau(\pi ; \xi) 
& = \tau(\pi ; \xi^f) \\
& \ge T(a,b ; \xi^f) \\
& = \widetilde\tau(\widetilde\pi ; \xi^f) \text{ as } \widetilde\pi \text{ is a }\xi^f\text{-good geodesic} \\
& \ge \widetilde\tau(\widetilde\pi ; \xi) \text{ as } \xi^f \subset \xi.
\end{align*}
The lemma is proven.
\end{proof}

We will need to study geodesics and travel time of geodesics.
Thanks to Lemma \ref{l:good-geodesic}, we can work with good geodesics.
For good geodesics, thanks to Lemma \ref{l:good}, the travel time $\tau$ is equal to the local travel time $\widetilde\tau$.
We can thus work with local travel times and local travel times are easier to handle.

\paragraph{$\rho$-skeleton associated with a path.}

Let $\pi$ be a path starting from $0$.
We see $\pi$ as a curve $[0,\ell(\pi)] \to \R^d$  parametrized by arc length.
We associate with $\pi$ a $\rho$-skeleton $\pi_\rho$ as follows. 
We first set $t_0=0 \in [0,\ell(\pi)]$ and $a_0=0 \in \R^d$ and then proceed by induction.
If $t_i$ and $a_i$ are defined for a given $i$, there are two cases:
 \begin{itemize}
  \item After time $t_i$, the path $\pi$ stays inside $B(a_i,\rho)$. In this case we set $k=k(\pi,\rho)=i$ and the construction is over.
  \item Otherwise, we denote by $t_{i+1}$ the first time after $t_i$ at which the path crosses 
  the sphere $S(a_i,\rho)$ of radius $\rho$ centered at $a_i$. 
  We set $a_{i+1}=\pi(t_{i+1})$ and the construction goes on.
 \end{itemize}
The $\rho$-skeleton $\pi_\rho$ is the path $(a_0,\dots,a_k)$.
Its length satisfies the inequality
\begin{equation} \label{e:length-skeleton}
\ell(\pi_\rho) = k\rho \ge \left\lfloor \frac{\|x_n\|}\rho \right\rfloor \rho \ge \|x_n\|-\rho.
\end{equation}

\subsection{Main result}
\label{s:main-controle-longueur}

\paragraph{Statement.}
The main result of Section \ref{s:controle-longueur} is the following one.

 \begin{theorem} \label{t:controle-l}
Let $\hat\nu$ be a finite measure on $(0,+\infty)$ satisfying the greedy condition \eqref{e:greedy}.
Assume $\mu(\hat\nu)>0$.
There exists $C=C(d,\hat\nu)$ such that the following holds.
Let $\nu \preceq \hat\nu$ be a measure on $(0,+\infty)$.
Let $\xi$ be a Poisson point process driven by $\nu$.
Then, almost surely,
\[
\limsup_{s \to \infty} \left[\sup\left\{\frac{\ell(\pi)}{s}, \; \pi \text{ is a } \xi \text{-good geodesic from } 0 \text{ to } S(0,s)\right\} \right]
\le 
C.
\]
\end{theorem}

\paragraph{Remark on the hypothesis $\mu (\tilde \nu) >0$.}
We want to enlighten the fact that the condition $\mu (\tilde \nu) >0$ appearing in Theorem \ref{t:controle-l} is quite natural. For that purpose, let us say a few words about the classical model of first-passage percolation on $\mathbb{Z}^d$. In this model, a time constant $\mu$ can also be defined by subadditivity (see for instance \cite{AuffingerDamronHanson} for a review on the subject), and it is known that $\mu>0$ if and only if $F(\{0\}) < p_c(d)$, where $F$ denotes the distribution of the passage times associated with the edges of $\mathbb{Z}^d$, and $p_c(d)$ is the critical parameter of i.i.d. bond Bernoulli percolation on $\mathbb{Z}^d$. The length $L_n$ of a geodesic between $0$ and $n e_1 = (n, 0,\dots, 0)$ in this model have been studied separately in the three following cases:
\begin{itemize}
\item $F(\{0\}) < p_c(d)$, {\em i.e.}, $\mu >0$ :  in this case, it has been proved by Kesten in \cite{Kesten-saint-flour} that $L_n$ is a most of order $C n$ for some constant $C=C(d,F)$. 
\item $F(\{0\}) > p_c(d)$ : then $\mu=0$, and some subadditivity can be recovered for a variant of $L_n$. This allowed Zhang and Zhang \cite {ZhangZhang} to prove that $L_n / n$ converges a.s. and in $L^1$ to some constant $C'=C'(d,F(\{0\})$ when $n$ goes to infinity in this setting (see also \cite{ZHANGsurcritique} for an improvement of the previous result).
\item $F(\{0\}) = p_c(d)$ : then $\mu =0$ also but it is believed that $L_n$ is superlinear in $n$. Only partial results of this type are rigorously proved, see for instance Damron and Tang's paper \cite{DamronTang} that proves indeed the superlinearity of $L_n$ in $n$ in dimension $2$.
\end{itemize}
We also refer to the paper by Bates \cite{bates2020empirical} for recent developments on the subject. Let us go back to the continuous model of first-passage percolation we are studying here. Theorem \ref{t:controle-l} requires the same kind of hypothesis as Kesten's result, {\it i.e.}, the underlying Boolean model must be subcritical. Kesten's proof is indeed a source of inspiration we widely use here. The study of the supercritical case ({\it i.e.}, when the underlying Boolean model is supercritical) should be accessible but quite different, whereas the study of the critical case is expected to be much more delicate.

\paragraph{Plan of the proof.} For some $\rho>0$, write
\[
\frac{\ell(\pi)}{s} = \frac{\ell(\pi)}{\ell(\pi_\rho)} \frac{\ell(\pi_\rho)}{\tau(\pi)} \frac{\tau(\pi)} s
\]
where $\pi_\rho$ is the $\rho$-skeleton of $\pi$.
We prove an upper bound for each of the three factors.

\begin{itemize}
\item Upper bound for $\ell(\pi)/\ell(\pi_\rho)$.
In Proposition \ref{p:controlellrho} we give a crude upper bound of the factor $\ell(\pi)/\ell(\pi_\rho)$.
The idea is the following.
Let $\pi_\rho=(a_0,\dots,a_k)$ be the $\rho$-skeleton of a path $\pi$.
Note that the distance between any two points of any ball or radius $\rho$ is at most $2\rho$.
Therefore, we can check that the length of $\pi$ is at most 
\[
2\rho N(a_0,\dots,a_k)+2(k+1)\rho 
\]
where $N(a_0,\dots,a_k)$ is the number of points of $\chi \cap \cup_i B(a_i,\rho)$.
As $\ell(\pi_\rho)=k\rho$, we have
\[
\frac{\ell(\pi)}{\ell(\pi_\rho)} \le 2\frac{k+1}k + 2\frac {N(a_0,\dots,a_k)} k
\]
and it remains to control the behavior of the supremum of $N(a_0,\dots,a_k)/k$ over all $\rho-$skeleton $(a_0,\dots,a_k)$ when $k$ tends to infinity.
\item Upper bound on $\ell(\pi_\rho)/\tau(\pi)$. Equivalently, we need to provide a good lower bound for its inverse
\[
\frac{\tau(\pi)}{\ell(\pi_\rho)}.
\]
This is the core of the proof. This is achieved in Proposition \ref{p:tau-ltaurho}.
Recall that the points $a_i$ of the skeleton $\pi_\rho=(a_0,\dots,a_k)$ are points of $\pi$.
We can thus think of $\pi$ as a union of subpaths: 
$\gamma^0$ the subpath of $\pi$ from $a_0$ to $a_1$, 
$\gamma^1$ the subpath of $\pi$ from $a_1$ to $a_2$ and so on.
As moreover $\ell(\pi_\rho)=k\rho$, we can write
\[
\frac{\tau(\pi)}{\ell(\pi_\rho)}  = \frac{\sum_{i=0}^k \tau(\gamma^i)}{k\rho} = \frac  1 k \sum_{i=0}^{k} \frac{\tau(\gamma^i)}{\rho}.
\]
But each $\gamma^i$ (except $\gamma^k$) is a path between two points at distance $\rho$ from each other.
Therefore one can hope to prove that $\tau(\gamma^i)/\rho$ is roughly at least the time constant $\mu$.
The basic plan to prove such a result is to use BK inequality.
However, one can only use BK inequality if the $\tau(\gamma^i)$ use distinct balls of the Boolean model.
But this is not the case for several reasons, the most important one being the existence of very large balls which touch several $\gamma^i$.
We thus have to deal with large balls before being able to use BK inequality.
This is achieved by bounding the influence of large balls by some greedy paths estimates.
This reduction of Proposition \ref{p:tau-ltaurho} to BK inequality and greedy paths estimates is performed in Proposition \ref{p:reduction}.

\item Upper bound on $\tau(\pi)/s$. This is the easiest part. 
Recall that $\pi$ is a geodesic from $0$ to $S(0,s)$.
Therefore $\tau(\pi)/s = T(0,S(0,s))/s$ and thus converges to the time constant $\mu$ by Theorem \ref{t:timeconstant}.
\end{itemize}

\paragraph{A by product of the proof: control of the length of the skeleton of a good geodesic.}
As a consequence of Items 2 and 3 of the plan, we obtain the following result which we believe is of independent interest.

\begin{theorem} \label{t:controle-ltaurho:simple}
Let $\nu$ be a finite measure on $(0,+\infty)$ satisfying the greedy condition \eqref{e:greedy}.
Let $\xi$ be a Poisson point process driven by $\nu$. 
Assume $\mu(\nu)>0$.
For all $\epsilon>0$ there exists $\rho_0=\rho_0(\nu, \epsilon,d)$ such that, for all $\rho \ge \rho_0$, almost surely,
\[
\limsup_{s \to \infty} \left[\sup\left\{\frac{\ell(\pi_\rho)}{s}, \; \pi \text{ is a good geodesic from } 0 \text{ to } S(0,s)\right\} \right]\le 
1+\epsilon.
\]
\end{theorem}

For any geodesic $\pi$ from $0$ to $S(0,s)$,
\[
\frac{\ell(\pi_\rho)}{s} \ge \frac{s-\rho}s = 1 - \frac{\rho}s.
\]
Therefore, by Theorem \ref{t:controle-ltaurho:simple}, for all $\rho$ large enough, almost surely, for all $s$ large enough, 
$\ell(\pi_\rho)/s$ is close to $1$ for any good geodesic $\pi$ from $0$ to $S(0,s)$.
Note that this is stronger that simply saying that $\pi$ is close to a segment after normalization by the distance between its extremities.

The result is not surprising.
Here is a heuristic (we rephrase Item 2 and Item 3 of the plan above).
Let $\pi$ be a geodesic from $0$ to $S(0,s)$.
Denote by $(a_0,\dots,a_k)$ its $\rho$-skeleton and decompose accordingly $\pi$ in subpaths $\pi^0, \dots, \pi^k$ as in the second item of the plan.
Then one can expect
\begin{align}
\mu s 
 & \approx \tau(\pi) \text{ because } \pi \text{ is a geodesic between } 0 \text{ and } S(0,s) \nonumber \\
 & = \sum_{i=0}^k \tau(\pi^i) \nonumber \\
 & = \sum_{i=0}^{k-1} T(a_i,a_{i+1}) + \tau(\pi^k) \text{ because each } \tau(\pi^i) \text{ is a geodesic}\nonumber  \\
 & \approx \sum_{i=0}^{k-1} \rho \mu \text{ as } \rho \text{ is large} \label{e:diff} \\
 & = \mu \ell(\pi_\rho). \nonumber
\end{align}
From $\mu s \approx \mu \ell(\pi_\rho)$ we get $s \approx \ell(\pi_\rho)$.
The difficulty is in \eqref{e:diff} and is mainly due to large balls which induce long range dependence, as explained with more details 
in the plan of the proof of Theorem \ref{t:controle-l}

\paragraph{Organization of the remaining of Section \ref{s:controle-longueur}.}
Section \ref{s:greedy} is devoted to greedy paths.
In Section \ref{s:p:controlellrho} we prove Proposition \ref{p:controlellrho}. This is Item 1 of the plan of the proof of Theorem \ref{t:controle-l}.
In Section \ref{s:reduction} we prove Proposition \ref{p:reduction}. This is the first step of Item 2 of the plan. 
We relate $\tau(\pi)/\ell(\pi_\rho)$ to a quantity $T^\square$ amenable to the use of BK inequality and to a quantity related to greedy paths.
In Section \ref{s:Tsquare} we study $T^\square$.
In Section \ref{s:tau-ltaurho} we prove Proposition \ref{p:tau-ltaurho} which gives a lower bound on $\tau(\pi)/\ell(\pi_\rho)$. 
This ends Item 2 of the plan.
In Section \ref{s:preuve:t:controle-l} we finally prove Theorem \ref{t:controle-l} and Theorem \ref{t:controle-ltaurho:simple}.

\subsection{Greedy paths}
\label{s:greedy}

Let $\nu$ be a finite measure on $(0,+\infty)$. 
Let $\xi$ be a Poisson point process on $\R^d \times (0,+\infty)$ with intensity measure $dc \times \nu(dr)$.
As before, we write
\[
\xi = \{(c,r(c)), c \in \chi\}.
\]
If $\pi=(x_0,\dots,x_n)$ is a path we set
\[
r(\pi) = r(\pi;\xi) = \sum_{i=0}^n r(x_i) \text{ and } \ell(\pi) = \sum_{i=1}^{n} \|x_i-x_{i-1}\|.
\]
We also set
\[
G = G(\xi) = \sup_{\pi : 0 \to *} \frac{r(\pi)}{\ell(\pi)} 
\]
where the supremum is taken over all paths $\pi=(x_0,\dots,x_n), n \ge 1$, such that $x_0=0$.
If moreover $s>0$, we write
\[
G(s)  = G(s;\xi) = \sup_{\pi : 0 \to *, \pi \not\subset B_o(0,s)} \frac{r(\pi)}{\ell(\pi)}
\]
where the supremum is taken over all paths $\pi=(x_0,\dots,x_n), n \ge 1$, such that $x_0=0$ and at least one of the $x_i$ is outside 
the open ball $B_o(0,s)$
As $G$ is non-increasing in $s$, we can define
\begin{equation}\label{e:ginfty}
G(\infty) = G(\infty;\xi) = \lim_{s \to \infty} G(s).
\end{equation}

\begin{theorem}[\cite{GM-deijfen}] \label{t:greedy} Let $\nu$ be a finite measure on $(0,+\infty)$.  There exists a constant $C=C(d)$ such that
\[
\E(G) \le C \int_{(0,+\infty)} \nu\big([r,+\infty)\big)^{1/d} dr.
\]
\end{theorem}
This is a consequence of (11) in \cite{GM-deijfen} and Lemma 2.1 in the same article.
Note that the results requires the assumption $d \ge 2$.
The result is the analogue in the continuous setting of a result by Martin \cite{Martin-greedy} in the discrete setting.

\begin{corollary} \label{c:greedy} Let $\nu$ be a finite measure on $(0,+\infty)$.  Then $G(\infty)$ is constant almost surely. Moreover, there exists a constant $C=C(d)$ such that
\[
G(\infty) \le C \int_{(0,+\infty)} \nu\big([r,+\infty)\big)^{1/d} dr.
\]
\end{corollary}

\begin{proof}
For any $r>0$, we have $G(\infty ; \xi) = G(\infty ; \xi \cap B(0,r)^c \times (0,+\infty))$. 
Therefore, by $0-1$ law, $G(\infty;\xi)$ is almost surely constant. 
Therefore, using $G(\infty) \le G$ and Theorem \ref{t:greedy}, we get a constant $C=C(d)$ such that
\[
G(\infty)=\E[G(\infty)] \le \E(G) \le C \int_{(0,+\infty)} \nu\big((r,+\infty)\big)^{1/d} dr.
\]
The corollary is proven.
\end{proof}

\subsection{Upper bound on $\ell(\pi)/\ell(\pi_\rho)$}
\label{s:p:controlellrho}

The aim of Section \ref{s:p:controlellrho} is to prove the following result.

\begin{prop} \label{p:controlellrho}
Let $\lambda>0$. 
Let $\chi$ be a Poisson point process on $\R^d$ with intensity $\lambda$ times the Lebesgue measure.
There exists $C=C(d)$ such that, for any $\rho>0$,
\[
\lim_{k \to \infty} \left[\sup\left\{ \frac{\ell(\pi)}{\ell(\pi_\rho)}, \; \pi \text{ a } \chi\text{-regular path from }0\text{ such that }\ell(\pi_\rho) \ge k\rho\right\} \right]\le \max(1,\lambda \rho^d) C. 
\]
\end{prop}

Note that $\ell(\pi)/\ell(\pi_\rho)$ is always greater or equal to $1$. Therefore the upper-bound cannot be of the form $\lambda\rho^d C$. 
Recall that we very quickly sketched the proof when announcing the plan of the proof of Theorem \ref{t:controle-l}.

\begin{proof}
By scaling, it is sufficient to prove the result when $\rho=1$.
We henceforth assume $\rho=1$ and aim at showing the existence of $C=C(d)$ such that
\[
\lim_{k \to \infty} \left[\sup\left\{ \frac{\ell(\pi)}{\ell(\pi_1)}, \; \pi \text{ a } \chi\text{-regular path from }0\text{ such that }\ell(\pi_1) \ge k\right\} \right]\le \max(1,\lambda) C. 
\]
Using the standard coupling, we get that the left hand-side is non-decreasing in $\lambda$.
It is therefore sufficient to prove the result when $\lambda \ge 1$.
We henceforth assume $\lambda\geq 1$ and aim at showing the existence of $C=C(d)$ such that
\[
\lim_{k \to \infty} \left[\sup\left\{ \frac{\ell(\pi)}{\ell(\pi_1)}, \; \pi \text{ a } \chi\text{-regular path from }0\text{ such that }\ell(\pi_1) \ge k\right\} \right]\le\lambda C. 
\]
Let $\pi=(x_0,\dots,x_n)$ be a regular path such that $x_0=0$.
Write $\pi_1=(a_0,\dots,a_k)$ and assume $\ell(\pi_1)=k \ge 1$.
For each $j \in \{1,\dots,k\}$, write 
\[
\gamma^j=\left(y^j_0,\dots,y^j_{n(j)}\right)
\]
for the subpath of $\pi$ between $a_{j-1}$ and $a_j$ and
\[
\gamma^{k+1}=\left(y^{k+1}_0,\dots,y^{k+1}_{n(k+1)}\right)
\]
for the subpath of $\pi$ after $a_k$.
With this notations,
\begin{equation}\label{e:sum}
\ell(\pi) = \sum_{j=1}^{k+1} \sum_{i=1}^{n(j)} \left\|y^j_i-y^j_{i-1}\right\|.
\end{equation}
For any $j \in  \{1,\dots,k+1\}$ and any $i \in \{1,\dots,n(j)\}$,
\[
\left\|y^j_i-y^j_{i-1}\right\| \le 2
\]
(recall $\rho=1$) as both points belong to the same ball of radius $1$.
The points of the family
\[
\left(y^j_i\right)_{j \in \{1,\dots,k+1\} \text{ and } i \in \{1,\dots,n(j)-1\}}
\]
are distinct points of $\chi$.
Moreover any such $y_i^j$ belongs to $B(y_0^j,1) = B(a_{j-1},1)$.
The number of term of the sum is therefore bounded by 
\[
\card\left(\chi \cap S(a_0,\dots,a_k)\right) + k+1
\]
where
\[
S(a_0,\dots,a_k) = \bigcup_{j=0}^k B(a_j,1).
\]
As a consequence,
\[
\ell(\pi)\le 2k+2+2\card\left(\chi \cap S(a_0,\dots,a_k)\right).
\]
On the other hand,
\[
\ell(\pi_\rho)= k.
\]
Therefore, it remains to prove the existence of a constant $C_1=C_1(d)$ such that 
\begin{equation}\label{e:intermediaire-lrhol}
\limsup_{k \to \infty} \left[\sup\left\{ \frac{\card\left(\chi \cap S(a_0,\dots,a_k)\right)}{k}, \; (a_0,\dots,a_k) \text{ is a } 1-\text{skeleton from }0\right\} \right]\le \lambda C_1. 
\end{equation}

Let $\kappa = \kappa(d) \geq 1$ such that $B(a, \kappa) \cap B(b, \kappa) \cap \Z^d \neq \emptyset$ for any $a,b\in \R^d$ satisfying $\|a-b\|=1$.
Let $(a_0,\dots,a_k), k \ge 1$, be a $1$-skeleton from $0$. 
Define
\[
S'(a_0,\dots,a_k) = \bigcup_{j=0}^k B(a_j,\kappa)
\]
and
\[
A(a_0,\dots,a_k) = \{x \in \Z^d : S'(a_0,\dots,a_k) \cap (x + [-1/2,1/2)^d) \neq \emptyset\}.
\]
We have 
\[
S(a_0,\dots,a_k) \subset  S'(a_0,\dots,a_k) \subset \bigcup_{x \in A(a_0,\dots,a_k) } (x+[-1/2,1/2)^d)
\]
and therefore
\[
\card\left(\chi \cap S(a_0,\dots,a_k)\right) \le \card\left(\chi \cap \bigcup_{x \in A(a_0,\dots,a_k) } (x+[-1/2,1/2)^d)\right).
\]
The random variable on the right hand-side is a Poisson random variable of parameter $\lambda \,\card A(a_0,\dots,a_k) $.
The set $A(a_0,\dots,a_k)$ contains $0$ and is a connected subset of $\Z^d$ for the usual graph structure on $\Z^d$. Indeed, $B(a_j, \kappa) \cap \Z^d$ is connected for any $j\in \{ 0,\dots , k \}$ and $B(a_{j-1}, \kappa) \cap B(a_j, \kappa) \cap \Z^d \neq \emptyset$ for any $j\in \{ 1,\dots , k \}$. In other words, $A(a_0,\dots,a_k)$ is an animal of $\Z^d$.
Moreover, its cardinality is at most $C_2k$ for some constant $C_2=C_2(d)$.
Let us denote by $\cA_m$ the set of animals of $\Z^d$ ({\it i.e.}, connected subsets of $\Z^d$ containing $0$) of cardinality at most $m$.
By the previous remarks we get, for any fixed $k \ge 1$,
\begin{align*}
& \sup\left\{\card\left(\chi \cap S(a_0,\dots,a_k)\right), \;  (a_0,\dots,a_k) \text{ is a } 1-\text{skeleton from }0\right\} \\
& \le \sup_{A \in \cA_{C_2 k}} \card\left(\chi \cap \bigcup_{x \in A} (x+[-1/2,1/2)^d)\right).
\end{align*}
There exists a constant $C_3=C_3(d)$ such that, for all $m \ge 1$,
\[
\card\cA_m \le C_3^m.
\]
See for example (4.24) in \cite{Grimmett-percolation}.
The previous supremum is therefore a supremum over at most $C_3^{C_2 k}$ Poisson random variables of parameter at most $C_2 k \lambda $.
We then get, for any $M>0$ and any $k$,
\begin{align*}
& \P\left[\sup\left\{\card\left(\chi \cap S(a_0,\dots,a_k)\right), \;  (a_0,\dots,a_k) \text{ is a } 1-\text{skeleton from }0\right\}  \ge M C_2 k \lambda \right] \\
& \le \P\left[ \sup_{A \in \cA_{C_2 k}} \card\left(\chi \cap \bigcup_{x \in A} (x+[-1/2,1/2)^d)\right) \ge M C_2 k \lambda \right] \\
& \le C_3^{C_2 k} \P\left[ \cP(C_2 k \lambda )\ge MC_2 k \lambda \right]
\end{align*}
where $\cP(C_2 k \lambda )$ denotes a Poisson random variable with parameter $C_2 k \lambda$.
Using classical Chernov bounds for Poisson random variables\footnote{
If $X$ is a Poisson random variable with parameter $u>0$ and if $M>1$ then
\[
\P[X \ge Mu] \le \exp\big[-u(1-M+M\ln(M))\big] = \left(\frac e M\right)^{Mu}e^{-u}.
\]
}, we then get, for any $M>1$ and any $k$,
\begin{align*}
& \P\left[\sup\left\{\card\left(\chi \cap S(a_0,\dots,a_k)\right), \;  (a_0,\dots,a_k) \text{ is a } 1-\text{skeleton from }0\right\}  \ge M C_2 k \lambda \right] \\
& \le C_3^{C_2 k} \left(\frac e M\right)^{M C_2 k \lambda} e^{-C_2 k \lambda} \\
& \le C_3^{C_2 k} \left(\frac 1 2\right)^{M C_2 k} \text{ if } M \ge 2e \text{ using also } \lambda \ge 1 \\
& = \left(\frac{C_3}{2^M}\right)^{C_2k}.
\end{align*}
Chose $M=M(d) \ge 2e$ such that
\[
\left(\frac{C_3}{2^M}\right)^{C_2} \le \frac 1 2.
\]
Then \eqref{e:intermediaire-lrhol} holds with $C_1=C_1(d)=MC_2$. 
\end{proof}

\subsection{Reduction of the study of $\tau(\pi)/\ell_\rho(\pi)$ to the study of $T^\square$ and $G$}
\label{s:reduction}

\subsubsection{Framework and result}

Let $\xi$ be a Poisson point process driven by a finite measure $\nu$ on $(0,\infty)$.
We write as usual $\xi=\{(c,r(c)), c \in \chi\}$.
Let $\rho>0$.
We say that a sequence $(a_0,\dots,a_k), k \ge 1$, is a $\rho$-skeleton if $a_0=0$ and, for all $j \in \{1,\dots,k\}$,
\[
\|a_j-a_{j-1}\|=\rho.
\]
We set
\[
T^\square(a_0,\dots,a_k) = \inf \left\{ \sum_{j=1}^k T\left(a_{j-1},a_j ; \xi^j\right) : \xi^1, \dots, \xi^k \text{ are disjoint finite subsets of } \xi \right\}.
\]
In words, this is the infimum\footnote{
See Section \ref{s:mesurabilite} in Appendix for remarks on the measurability of $T^\square$.}
of the sum over $j$ of the travel times from $a_{j-1}$ to $a_j$ where one is not allowed to use the same ball of the Boolean model for two different $j$.

For any $r_0>0$ we set
\[
\xi^{> r_0} = \{(c, r(c)), c \in \chi \text{ such that } r(c) > r_0\}.
\]
We will freely use this kind of notation for different point processes.

The aim of Section \ref{s:reduction} is to prove the following result.
We couple the two point processes as explained in the paragraph "Domination and coupling" of Section \ref{s:couplings}.

\begin{prop} \label{p:reduction} 

Let $\nu$ and $\hat\nu$ be two finite measures on $(0,+\infty)$. 
We assume that $\nu$ is dominated by $\hat\nu$.
Let $\xi$ and $\hat\xi$ be two coupled Poisson point processes driven by $\nu$ and $\hat\nu$.
Let $\rho, \delta >0$ and $s \ge \rho$.
Let $\pi=(x_0,\dots,x_n)$ be a path from $0$ to $B(0,s)^c$.
Then
\begin{align*}
\frac{\widetilde\tau(\pi;\xi)}{\ell(\pi_\rho)} 
  & \ge  \inf \left\{ \frac 1 {k \rho} T^\square\left(a_0,\dots,a_k;\hat{\xi} \right) :  k \ge \lfloor s/\rho \rfloor \text{ and } (a_0,\dots,a_k) \text{ is a }\rho\text{-skeleton from }0\right\}   \\
  & - (2+2\delta^{-1}) G\left(s-\rho; \hat{\xi}^{>\delta\rho}\right) - 2\delta.
\end{align*}
\end{prop}

\subsubsection{Proof of Proposition \ref{p:reduction}}

Notice that Lemmas \ref{l:cutting} and \ref{l:prepare-BK} are actually purely deterministic.

\begin{lemma} \label{l:cutting} Let $\rho>0$. Let $\pi$ be a path and let $k$ and $(t_0,\dots,t_k)$ be defined as in the construction of the $\rho$-skeleton of $\pi$ (see section \ref{s:skeleton}).
Cutting $\pi$ (seen as a parametrized curve) 
at each $t_i$ and throwing away the part after time $t_k$, if any, we get $k$ sub-paths $\gamma^1, \dots, \gamma^k$.
 Let $M>0$. Assume, for all $c \in \chi, r(c) \le M$.
 Then
 \[
 \widetilde\tau(\pi) \ge \sum_{j=1}^k \widetilde\tau\left(\gamma^j\right) - 2Mk.
 \]
\end{lemma}
\begin{proof} The error term $2Mk$ is due to boundary effects related to the cutting at each $a_j$ and to the definition of $\widetilde\tau$.
For each $j \in \{1,\dots,k\}$, write
\[
\gamma^j=\left(y^j_0,\dots,y^j_{n(j)}\right).
\]
With appropriately defined passage times, we can write
\[
\sum_{j=1}^k \widetilde\tau\left(\gamma^j\right) = \sum_{j=1}^k \sum_{i=1}^{n(j)} A^j_i.
\]
Here, $A^j_i$ denotes the contribution of the segment $[y^j_{i-1},y^j_i]$ to the sum of the travel times.
Note that it does not only depend on $y^j_{i-1}$ and $y^j_i$.
It also depends on whether $i-1=0$ and on whether $i=n(j)$.

In the same way, disregarding what happens after time $t_k$ in the path $\pi$, we can write
\[
\widetilde\tau(\pi) \ge \sum_{j=1}^k \sum_{i=1}^{n(j)} B^j_i
\]
where the $B^j_i$ are also appropriately defined passage times.
Here, $B^j_i$ corresponds to the contribution of the segment $[y^j_{i-1},y^j_i]$ to $\widetilde\tau (\pi)$, but it may differ from $A^j_i$.
Let $j \in \{1,\dots,k\}$ and $i \in \{1,\dots,n(j)\}$.
One and only one of the following cases occurs (see Figure \ref{f:fig3}).
\begin{figure}[!ht]
\centering
\begingroup%
  \makeatletter%
  \providecommand\color[2][]{%
    \errmessage{(Inkscape) Color is used for the text in Inkscape, but the package 'color.sty' is not loaded}%
    \renewcommand\color[2][]{}%
  }%
  \providecommand\transparent[1]{%
    \errmessage{(Inkscape) Transparency is used (non-zero) for the text in Inkscape, but the package 'transparent.sty' is not loaded}%
    \renewcommand\transparent[1]{}%
  }%
  \providecommand\rotatebox[2]{#2}%
  \newcommand*\fsize{\dimexpr\f@size pt\relax}%
  \newcommand*\lineheight[1]{\fontsize{\fsize}{#1\fsize}\selectfont}%
  \ifx\svgwidth\undefined%
    \setlength{\unitlength}{331.78948098bp}%
    \ifx\svgscale\undefined%
      \relax%
    \else%
      \setlength{\unitlength}{\unitlength * \real{\svgscale}}%
    \fi%
  \else%
    \setlength{\unitlength}{\svgwidth}%
  \fi%
  \global\let\svgwidth\undefined%
  \global\let\svgscale\undefined%
  \makeatother%
  \begin{picture}(1,0.32862243)%
    \lineheight{1}%
    \setlength\tabcolsep{0pt}%
    \put(0,0){\includegraphics[width=\unitlength,page=1]{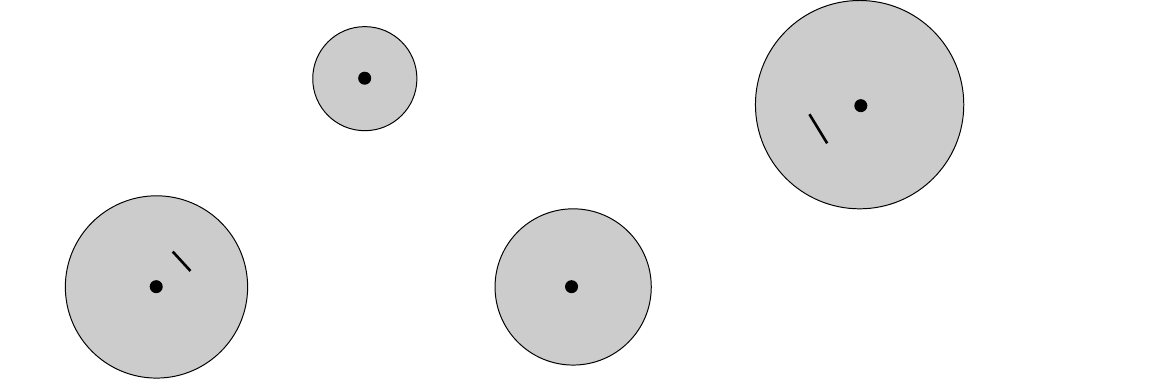}}%
    \put(0.28738091,0.27296716){\color[rgb]{0,0,0}\makebox(0,0)[lt]{\lineheight{1.25}\smash{\begin{tabular}[t]{l}$y_1^j$\end{tabular}}}}%
    \put(0.71802979,0.19441231){\color[rgb]{0,0,0}\makebox(0,0)[lt]{\lineheight{1.25}\smash{\begin{tabular}[t]{l}$a_{j+1}=y_3^j$\end{tabular}}}}%
    \put(0.46272653,0.0429137){\color[rgb]{0,0,0}\makebox(0,0)[lt]{\lineheight{1.25}\smash{\begin{tabular}[t]{l}$y_2^j$\end{tabular}}}}%
    \put(0.00682792,0.17056531){\color[rgb]{0,0,0}\makebox(0,0)[lt]{\lineheight{1.25}\smash{\begin{tabular}[t]{l}$\pi$\end{tabular}}}}%
    \put(0.24670073,0.12567683){\color[rgb]{0,0,0}\makebox(0,0)[lt]{\lineheight{1.25}\smash{\begin{tabular}[t]{l}$B_1^j$\end{tabular}}}}%
    \put(0.14710442,0.22387039){\color[rgb]{0,0,0}\makebox(0,0)[lt]{\lineheight{1.25}\smash{\begin{tabular}[t]{l}$A_1^j$\end{tabular}}}}%
    \put(0.6324611,0.10323259){\color[rgb]{0,0,0}\makebox(0,0)[lt]{\lineheight{1.25}\smash{\begin{tabular}[t]{l}$B_3^j$\end{tabular}}}}%
    \put(0.54268414,0.17757913){\color[rgb]{0,0,0}\makebox(0,0)[lt]{\lineheight{1.25}\smash{\begin{tabular}[t]{l}$A_3^j$\end{tabular}}}}%
    \put(0.42906017,0.21825933){\color[rgb]{0,0,0}\makebox(0,0)[lt]{\lineheight{1.25}\smash{\begin{tabular}[t]{l}$B_2^j=A_2^j$\end{tabular}}}}%
    \put(0,0){\includegraphics[width=\unitlength,page=2]{fig3.pdf}}%
    \put(0.1681459,0.0863994){\color[rgb]{0,0,0}\makebox(0,0)[lt]{\lineheight{1.25}\smash{\begin{tabular}[t]{l}$a_j = y_0^j$\end{tabular}}}}%
  \end{picture}%
\endgroup%
\caption{A portion of the path $\pi$ is represented in dashed line, and for a given $j$ we look at $\gamma^j = (y_0^j = a_{j-1}, \dots , y_3^j = a_j)$ (here $n(j)=3$). The passage time $B^j_i$ is the length of the solid thick line included in the segment $[y_{i-1}^j , y_i^j]$. The passage time $A^j_i$ is the length of the solid line - thick and thin - included in the segment $[y_{i-1}^j , y_i^j]$. The difference between $A^j_i - B^j_i$ is thus the length of the thin solid line included in the segment $[y_{i-1}^j , y_i^j]$, if any, and is smaller than $M$.}
\label{f:fig3}
\end{figure}
\begin{itemize}
  \item $0<i-1$ and $i<n(j)$. In this case $A_i^j=B_i^j$.
  \item $0=i-1$ and $i<n(j)$. In this case, as the radii are bounded above by $M$, $A_i^j \le B_i^j + M$.
  \item $0<i-1$ and $i=n(j)$. In this case, as above, $A_i^j \le B_i^j + M$.
  \item $0=i-1$ and $i=n(j)$. In this case, $A_i^j \le B_i^j + 2M$.
\end{itemize}
Therefore, for each $j$,
\[
\sum_{i=1}^{n(j)} A^j_i \le 2M + \sum_{i=1}^{n(j)} B^j_i
\]
and then
\[
\sum_{j=1}^k \sum_{i=1}^{n(j)} A^j_i \le 2Mk + \sum_{j=1}^k \sum_{i=1}^{n(j)} B^j_i.
\]
The lemma follows.
\end{proof}

\begin{lemma} \label{l:prepare-BK} Let $\rho>0$. 
Let $\xi$ be a Poisson point process driven by a finite measure $\nu$ on $(0,\infty)$.
Let $M>0$. Assume, for all $c \in \chi, r(c) \le M$.
Let $\pi$ be a path from $0$ to a point outside $B(0,\rho)$.
Denote by $(a_0,\dots,a_k), k \ge 1$, its $\rho$-skeleton.
 Then 
 \[
 \widetilde\tau(\pi) \ge T^\square(a_0,\dots,a_k)  - 2Mk.
 \]
\end{lemma}

\begin{proof} By Lemma \ref{l:cutting} we get
\[
 \widetilde\tau(\pi) \ge \sum_{j=1}^k \widetilde\tau\left(\gamma^j\right) - 2Mk.
\]
Let $j \in \{1,\dots,k\}$. Write 
\[
\gamma^j=\left(y^j_0,\dots,y^j_{n(j)}\right).
\]
Write
\[
\xi^j = \{ (y^j_i, r(y^j_i)), i \in \{1,\dots,n(j)-1\} \text{ such that }r(y^j_i)>0 \}.
\]
Note that we do not assume the paths to be regular, so there may be no ball centered at some $y^j_i$, this is why there is this condition "$r(y^j_i)>0$".
By definition of the paths, the $\xi^j$ are disjoint finite subsets of $\xi$.
By definition of $\tau$ and $\widetilde\tau$, for any $j \in \{1,\dots,k\}$,
\[
\widetilde\tau\left(\gamma^j\right) \ge \tau\left(\gamma^j ; \xi^j \right).
\]
Therefore
\[
\sum_{j=1}^k \widetilde\tau\left(\gamma^j\right)   \ge \sum_{j=1}^k \tau\left(\gamma^j ; \xi^j \right)
\]
and then
\[
\widetilde\tau(\pi) \ge  \sum_{j=1}^k \tau\left(\gamma^j ; \xi^j \right) -2Mk \ge \sum_{j=1}^k T\left(a_{j-1},a_j ;\xi^j \right) -2Mk
\]
By definition of $T^\square(a_0,\dots,a_k)$, the lemma follows.
\end{proof}

The following simple lemma is deterministic but we keep the notations of the random setting.
We define $\widetilde\tau(\cdot)$, $Z(\cdot)$ and so on as usual. We will use it in a deterministic setting 
and in a probabilistic setting where for example $\chi$ will be as usual the realization of a homogeneous Poisson point process on $\R^d$.

\begin{lemma} \label{l:modification-sigma} 
Let $\chi$ be a locally finite subset of $\R^d$.
Let $(r^-(c))_{c \in \chi}$ and $(r^+(c))_{c \in \chi}$ be two families of non-negative real numbers such that, for all $c$, $r^-(c) \le r^+(c)$.
Write $\xi^-=\{(c,r^-(c)), c \in \chi \}$ and $\xi^+=\{(c,r^+(c)), c \in \chi \}$.
Let $\pi=(x_0,\dots,x_n)$ be a path. Then
  \[
  \widetilde\tau(\pi,\xi^-) \le \widetilde\tau(\pi,\xi^+) + 2 \sum_{i \in \{1,\dots,n-1\}} \big(r^+(x_i)-r^-(x_i)\big).
  \]
\end{lemma}

\begin{proof} 
For all $i \in \{1,\dots,n-1\}$, 
$Z^+_i(\pi ; \xi^+)$ is included in the union of $Z^+_i(\pi ; \xi^-)$ and of a segment of length at most $r^+(x_i)-r^-(x_i)$.
A similar statement holds for $Z^-_i$. 
Recall that $Z^+_0(\pi ; \cdot) = Z^-_n(\pi ; \cdot) = \emptyset$.
Therefore,
\begin{align*}
\widetilde\tau(\pi ; \xi^-)
 &  = \sum_{i=1}^n \ell\left([x_{i-1},x_i] \setminus \left( Z^+_{i-1}(\pi ; \xi^-) \cup Z^-_i(\pi ; \xi^-) \right)\right) \\
 & = \ell(\pi) - \sum_{i=1}^n \ell\left([x_{i-1},x_i] \cap \left( Z^+_{i-1}(\pi ; \xi^-) \cup Z^-_i(\pi ; \xi^-) \right)\right) \\
 & \le \ell(\pi) - \sum_{i=1}^n \ell\left([x_{i-1},x_i] \cap \left( Z^+_{i-1}(\pi ; \xi^+) \cup Z^-_i(\pi ; \xi^+) \right)\right) 
 + 2 \sum_{i \in \{1,\dots,n-1\}} \big(r^+(x_i)-r^-(x_i)\big) \\
 & = \sum_{i=1}^n \ell\left([x_{i-1},x_i] \setminus \left( Z^+_{i-1}(\pi ; \xi^+) \cup Z^-_i(\pi ; \xi^+) \right)\right) 
 + 2 \sum_{i \in \{1,\dots,n-1\}} \big(r^+(x_i)-r^-(x_i)\big) \\
 & = \widetilde\tau(\pi ; \xi^+) + 2 \sum_{i \in \{1,\dots,n-1\}} \big(r^+(x_i)-r^-(x_i)\big).
\end{align*}
The lemma is proven.
\end{proof}

 \paragraph{The $(\rho,\delta)$ enhanced skeleton.}
 Let $\rho, \delta>0$. 
 Let $\pi$ be a path starting at $0$.
 Let $(a_0,\dots,a_k)$ be its associated $\rho$-skeleton.
 We define a sequence $s_\rho^\delta$ and a path $\pi_\rho^\delta$. 
 The definition of these objects is in a sense artificial, since it is built to enable the use of results on greedy paths.
 For any $x \in \R^d$, write
 \[
 D(x) = \{c \in \chi :  \delta^{-1} r(c) \ge \|c-x\|  \text{ and }r(c) > \delta\rho\}.
 \]
\begin{itemize}
 \item The sequence $s^\delta_\rho$ starts from $a_0 = 0$ and visits each point of $D(a_0)$ it has not visited yet
 (in an order given by any specified rule), if any.
 \item Then it goes to $a_1$ and visits each point of $D(a_1)$ it has not visited yet, if any.
 \item $\dots$
 \item Then it goes to $a_k$ and visits each point of $D(a_k)$ it has not visited yet, if any.
\end{itemize}
Recall that a path is a sequence of {\em distinct} points of $\R^d$. 
The sequence $s^\delta_\rho$ may not be a path because when going to some $a_i$, the point may have already been visited.
With $s^\delta_\rho$ we associate the path $\pi^\delta_\rho$ by erasing from $s^\delta_\rho$ each point which also occurs earlier in the sequence.
In other words, we only keep in $\pi^\delta_\rho$ the first occurrence of each point.
Let us prove
\begin{equation}\label{e:length-enhanced}
\ell(\pi_\rho) \le \ell(s^\delta_\rho) \le \ell(\pi_\rho) + 2 \delta^{-1} r(\pi^\delta_\rho ; \xi^{>\delta\rho}).
\end{equation}
The first inequality of \eqref{e:length-enhanced} is trivial.
The second inequality can be proven as follows.
Consider the longer sequence
in which, in the definition,
"visits each point of $D(a_i)$ it has not visited yet" is replaced by "goes back and forth between $a_i$ and points of $D(a_i)$ it has not visited yet".
The length of each round-trip between a point $a_i$ and a point  $c \in D(a_i)$ is at most $2\delta^{-1} r(c)$.
As moreover each $c \in D(a_i)$ is such that $r(c)>\delta\rho$,
the total length of all the round-trips is at most $2\delta^{-1} r(\pi^\delta_\rho ; \xi^{>\delta\rho})$.
Therefore the length of this new sequence is at most
\[
\ell(\pi_\rho) + 2 \delta^{-1} r(\pi^\delta_\rho; \xi^{>\delta\rho}).
\]
As this sequence is longer that $s^\delta_\rho$ this establishes \eqref{e:length-enhanced}.

\begin{proof}[Proof of Proposition \ref{p:reduction}] We begin by throwing away large balls. Set 
\[
\xi^{\le \delta \rho} = \{(c, r(c)), c \in \chi \text{ such that } r(c) \le \delta\rho\}
\]
and
\[
\xi^{\wedge \delta \rho} = \{(c, \min(r(c),\delta\rho)), c \in \chi\}.
\]
Using Lemma \ref{l:modification-sigma} and greedy paths notations, we get for any path $\pi = (x_0,\dots , x_n)$,
\begin{align*}
\widetilde\tau(\pi;\xi^{\le \delta\rho})
 & \le \widetilde\tau(\pi) + 2 \sum_{i=0}^n \left(r(x_i) - r(x_i)\1_{r(x_i) \le \delta\rho}\right) \\
 & = \widetilde\tau(\pi) + 2 \sum_{i=0}^n r(x_i)\1_{r(x_i) > \delta\rho} \\
 & = \widetilde\tau(\pi) + 2 r(\pi ; \xi^{>\delta\rho}).
\end{align*}
Therefore
\begin{equation}\label{e:hobbit}
\widetilde\tau(\pi) \ge \widetilde\tau(\pi;\xi^{\le \delta\rho}) - 2 r(\pi; \xi^{>\delta\rho}).
\end{equation}

With the path $\pi=(x_0,\dots,x_n)$ we now associate the $\rho$ skeleton $\pi_\rho=(a_0,\dots,a_k)$ 
and the $(\rho,\delta)$ enhanced skeletons $s_\rho^\delta$ (which may not be a path) and $\pi_\rho^\delta$.
Note that $k$ is at least $1$ because $\|x_n\| \ge s \ge \rho$.
For any $i \in \{0,\dots,n\}$, there exists $j \in \{0,\dots,k\}$ such that $\|x_i -a_j \| \le \rho$.
If moreover $r(x_i)>\delta\rho$, then $\|x_i -a_j \| \le \delta^{-1} r(x_i)$.
Thus $x_i$ is one of the points of $\pi_\rho^{\delta}$.
Therefore
\begin{equation}\label{e:enhanced-skeleton-r}
r\left(\pi ; \xi^{>\delta\rho}\right) \le r\left(\pi_\rho^{\delta};  \xi^{>\delta\rho}\right).
\end{equation}
Regarding travel times, by Lemma \ref{l:prepare-BK}, we have
\begin{equation} \label{e:feu}
\widetilde\tau(\pi;\xi^{\le \delta\rho} ) \ge T^\square(a_0,\dots,a_k ;\xi^{\le \delta\rho})-2k\delta\rho.
\end{equation}
By \eqref{e:hobbit},  \eqref{e:enhanced-skeleton-r} and \eqref{e:feu} we then get
\begin{equation}\label{e:hobbit2}
\widetilde\tau(\pi) \ge T^\square(a_0,\dots,a_k ;\xi^{\le \delta\rho})-2k\delta\rho - 2 r\left(\pi_\rho^{\delta};  \xi^{>\delta\rho}\right).
\end{equation}
Therefore
\begin{align*}
 \frac{\widetilde\tau(\pi)}{\ell(\pi_\rho)} 
& \ge 
\frac
{\widetilde\tau(\pi)}
{\ell\left(s_\rho^{\delta}\right)} 
\text{ by }  \eqref{e:length-enhanced},\\
& = 
\frac
{\widetilde\tau(\pi)+(2+2\delta^{-1})r\left(\pi_\rho^{\delta};  \xi^{>\delta\rho}\right)}
{\ell\left(s_\rho^{\delta}\right)} 
- 
\frac
{(2+2\delta^{-1})r\left(\pi_\rho^{\delta};  \xi^{>\delta\rho}\right)}
{\ell\left(s_\rho^{\delta}\right)}  \\
& \ge 
\frac
{\widetilde\tau(\pi)+(2+2\delta^{-1})r\left(\pi_\rho^{\delta};  \xi^{>\delta\rho}\right)}
{\ell\left(\pi_\rho\right)+2\delta^{-1}r\left(\pi_\rho^{\delta};  \xi^{>\delta\rho}\right)} 
- 
\frac
{(2+2\delta^{-1})r\left(\pi_\rho^{\delta};  \xi^{>\delta\rho}\right)}
{\ell\left(\pi_\rho^{\delta}\right)}
 \text{ by } \eqref{e:length-enhanced} \text{ and as } \ell\left(s_\rho^{\delta}\right) \ge \ell\left(\pi_\rho^{\delta}\right)\\
& \ge 
\frac
{T^\square(a_0,\dots,a_k ;\xi^{\le \delta\rho})-2k\delta\rho+ 2 \delta^{-1}r\left(\pi_\rho^{\delta}; \xi^{>\delta\rho}\right)}
{\ell\left(\pi_\rho\right)+2\delta^{-1}r\left(\pi_\rho^{\delta};  \xi^{>\delta\rho}\right)} 
- 
\frac
{(2+2\delta^{-1})r\left(\pi_\rho^{\delta};  \xi^{>\delta\rho}\right)}
{\ell\left(\pi_\rho^{\delta}\right)} \text{ by } \eqref{e:hobbit2}.
\end{align*}
As 
\[
T^\square(a_0,\dots,a_k ;\xi^{\le \delta\rho})-2k\delta\rho  \le T^\square(a_0,\dots,a_k ;\xi^{\le \delta\rho}) \le k\rho = \ell(\pi_\rho)
\]
we have\footnote{We use the fact that $(a+c)/(b+c) \ge a/b$ for any $a \le b$
and any $c \ge 0$ with $b>0$.}
\begin{align*}
 \frac{\widetilde\tau(\pi)}{\ell(\pi_\rho)} 
 & \ge 
 \frac
 {T^\square(a_0,\dots,a_k ;\xi^{\le \delta\rho})-2k\delta\rho }
 {\ell\left(\pi_\rho\right)}   
 - \frac
 {(2+2\delta^{-1}) r\left(\pi_\rho^{\delta}; \xi^{>\delta\rho}\right)}
 {\ell\left(\pi_\rho^{\delta}\right)}.
\end{align*}
Using \eqref{e:length-skeleton} we get $\ell(\pi_\rho)=k\rho$ and $k \ge \lfloor s/\rho \rfloor$.
Note that $\pi_\rho$, and therefore $\pi^\delta_\rho$, is not contained in $B(0,s-\rho)$ (this is because $\pi$ is not contained in $B(0,s)$).
Using skeleton notations and greedy paths notations, we get
\begin{align*}
 \frac{\widetilde\tau(\pi)}{\ell(\pi_\rho)} 
& \ge \inf \left\{ \frac 1 {k \rho} T^\square\left(a_0,\dots,a_k;\xi^{\le \delta\rho} \right) :  k \ge \lfloor s/\rho \rfloor \text{ and } (a_0,\dots,a_k) \text{ is a }\rho\text{-skeleton}\right\}  - (2+2\delta^{-1}) G(s-\rho; \xi^{>\delta\rho}) -2\delta\\
& \ge \inf \left\{ \frac 1 {k \rho} T^\square\left(a_0,\dots,a_k;\hat{\xi} \right) :  k \ge \lfloor s/\rho \rfloor \text{ and } (a_0,\dots,a_k) \text{ is a }\rho\text{-skeleton}\right\}  - (2+2\delta^{-1}) G(s-\rho; \hat{\xi}^{>\delta\rho}) -2\delta.
\end{align*}
The proposition is proven.
\end{proof}

\subsection{Control of $T^\square$}
\label{s:Tsquare}

We need a simple lemma to define the next object.

\begin{lemma} \label{l:definition-mu-square} 
Let $\xi$ be a Poisson point process driven by a finite measure $\nu$ on $(0,\infty)$. The dependence in $\xi$ is implicit in what follows.
Assume $\int_{(0,+\infty)} r^d \nu (\d r) < \infty$.
Then, with probability one,
\begin{equation}\label{e:musquare}
\mu^\square(\rho,\nu) : =  \lim_{k_0 \to \infty} \inf \left\{ \frac 1 {k \rho} T^\square(a_0,\dots,a_k) :  k \ge k_0 \text{ and } (a_0,\dots,a_k) \text{ is a }\rho\text{-skeleton from }0\right\}.
\end{equation}
exists and is deterministic.
\end{lemma}

\begin{proof}
Note that $k \rho$ is the total Euclidean length of the sequence $(a_0,\dots,a_k)$.
Therefore the random sequence
\[
k_0 \mapsto \inf \left\{ \frac 1 {k \rho} T^\square(a_0,\dots,a_k) :  k \ge k_0 \text{ and } (a_0,\dots,a_k) 
\text{ is a }\rho\text{-skeleton from }0\right\} =: I(k_0)
\]
takes values in $[0,1]$.
It is moreover measurable\footnote{See Section \ref{s:mesurabilite} in Appendix for remarks on the measurability of $I$.} and non-decreasing.
Therefore it converges when $k_0$ tend to $\infty$ toward a measurable limit that we denote by $\ell(\xi)$.

Let $e_1$ denote the first vector of the canonical basis of $\R^d$.
Let $k_0 \ge 1$.
For all $\rho$-skeleton $(a_0,\dots,a_k)$ from $0$ with $k \ge k_0$ we have, 
translating everything by $-\rho e_1$ in the first step\footnote{
We extend in a natural way the definition of $\rho$-skeleton to skeletons $(a_0,\dots, a_k)$ starting from a point $a_0 \neq 0$, and the definition of $T^\square (a_0,\dots , a_k)$ for such skeletons. 
}
and expliciting the dependence on 
the point process,
\begin{align*}
\frac 1 {k \rho} T^\square(a_0,\dots,a_k ; \xi) 
 & = \frac 1 {k \rho} T^\square(a_0-\rho e_1,\dots,a_k-\rho e_1 ; \xi-\rho e_1) \\
 & \ge \frac 1 {k \rho} \big(T^\square(a_0,a_0-\rho e_1,\dots,a_k-\rho e_1 ; \xi-\rho e_1) - \rho\big) \\
 & =  \frac 1 {(k +1)\rho} T^\square(a_0,a_0-\rho e_1,\dots,a_k-\rho e_1 ; \xi-\rho e_1) - \frac 1 k \\
 & \ge  I(k_0 ; \xi -\rho e_1)  - \frac 1 {k_0}.
\end{align*}
Taking infimum over all $k \ge k_0$ and all  $\rho$-skeleton $(a_0,\dots,a_k)$ from $0$ we get 
\[
I(k_0 ; \xi) \ge  I(k_0 ; \xi -\rho e_1)- \frac 1 {k_0}.
\]
By symmetry we have the reverse inequality and thus
\[
|I(k_0 ; \xi) - I(k_0 ; \xi -\rho e_1)| \le \frac 1 {k_0}.
\]
As a consequence, taking limit when $k_0$ tends to infinity, we have $\ell(\xi)=\ell(\xi-\rho e_1)$.
By ergodicity, $\ell$ is therefore constant almost surely.
\end{proof}

\begin{prop} \label{p:travel-time-skeleton} 
Let $\xi$ be a Poisson point process driven by a finite measure $\nu$ on $(0,\infty)$. 
Assume $\int_{(0,+\infty)} r^d \nu (\d r) < \infty$.
Denote by $\mu(\nu)$ the associated time constant.
Then 
\[
\liminf_{\rho \to +\infty} \mu^\square(\rho,\nu) \ge \mu(\nu).
\]
\end{prop}

\begin{proof}
Let $\epsilon>0$. We aim at proving that, for any $\rho>0$ large enough,
\[
\mu^\square(\rho,\nu) \ge \mu(\nu) - \epsilon.
\]
Let us assume $\mu(\nu)-\epsilon > 0$, otherwise this is straightforward.   
Write $\eta=\epsilon/8$.
We then have, as $\mu(\nu) \le 1$, 
\[
\mu(\nu)(1-2\eta) > \mu(\nu)-\epsilon/2. 
\]
Let $\alpha=\alpha(d)>0$ be such that
\[
\R^d = \bigcup_{x \in \alpha \Z^d} B(x,1).
\]
We then have, for any $\rho>0$,
\[
\R^d = \bigcup_{x \in \alpha \eta \rho \Z^d} B(x,\eta\rho).
\]
Define $C=C(\epsilon,d)$ by
\[
C=\card\big[\big(B(0,1+2\eta) \setminus B(0,1-2\eta) \big)\cap \alpha \eta \Z^d\big].
\]
We then have, for any $\rho>0$,
\[
C = \card\big[ \big(B(0,\rho(1+2\eta)) \setminus B(0,\rho(1-2\eta)) \big)\cap \alpha\eta\rho  \Z^d\big].
\]

Let $\rho>0$.
Let $(a_0,\dots,a_k)$ be a $\rho$-skeleton such that $a_0=0$. 
For any $j \in \{0,\dots,k-1\}$, there exists $b_j \in \alpha\eta\rho\Z^d$ such that 
\begin{equation} \label{e:ouestaj}
a_j \in B(b_j,\eta\rho).
\end{equation}
For $j=0$ we choose $b_0=0$.
Moreover, for any such $j$, as $\|a_{j+1}-a_j\|=\rho$,
\begin{equation} \label{e:ouestajp1}
a_{j+1} \in B(b_j,\rho-\eta\rho)^c
\end{equation}
and
\begin{equation} \label{e:ecartbjbjp1}
\rho(1-2\eta) \le \|b_{j+1}-b_j\| \le \rho(1+2\eta).
\end{equation}
By \eqref{e:ouestaj} and \eqref{e:ouestajp1} we get
\begin{align*}
T^\square(a_0,\dots,a_k) 
 & = \inf \left\{ \sum_{j=0}^{k-1} T\left(a_{j},a_{j+1};\xi^j\right) : \xi^0, \dots, \xi^{k-1} \text{ are disjoint finite subsets of } \xi \right\} \\
 & \ge \inf \left\{ \sum_{j=0}^{k-1} T\left(B(b_{j},\eta\rho),B(b_{j},\rho-\eta\rho)^c;\xi^j\right) : \xi^0, \dots, \xi^{k-1} 
 \text{ are disjoint finite subsets of } \xi \right\}.
\end{align*}
Therefore
\begin{align}
\label{e:ajoutM}
& \P\left[ \inf \left\{T^\square(a_0,\dots,a_k) : (a_0,\dots,a_k) \; \rho\text{-skeleton from }0\right\} < k\rho (\mu(\nu)-\epsilon) \right] \nonumber \\
& \le  \sum_{b_0,\dots,b_{k-1}} \P\left[ \inf \left\{ \sum_{j=0}^{k-1} T\left(B(b_{j},\eta\rho),B(b_{j},\rho-\eta\rho)^c;\xi^j\right) : 
\xi^0, \dots, \xi^{k-1} \text{ are disjoint finite subsets of } \xi \right\} < k\rho (\mu(\nu)-\epsilon)\right]
\end{align}
where the sum is taken over all $(b_0,\dots,b_{k-1})$ such that $b_0=0$ and, for all $j \in \{1,\dots,k-1\}$,
\[
b_j - b_{j-1} \in \big(B(0,\rho(1+2\eta)) \setminus B(0,\rho(1-2\eta)) \big)\cap \alpha\eta\rho \Z^d
\]
whose cardinality is $C$.	
There are therefore at $C^{k-1}$ such sequences.

\begin{claim} \label{c:bk} For all such $(b_0,\dots,b_{k-1})$,
\begin{align*}
&  \P\left[ \inf \left\{ \sum_{j=0}^{k-1} T\left(B(b_{j},\eta\rho),B(b_{j},\rho-\eta\rho)^c;\xi^j\right) : 
\xi^0, \dots, \xi^{k-1} \text{ are disjoint finite subsets of } \xi \right\} < k\rho (\mu(\nu)-\epsilon)\right] \\
& \le \P\left[\sum_{j=0}^{k-1} T^{(j)} < k\rho (\mu(\nu)-\epsilon) \right]
\end{align*}
where the $T^{(j)}$ are independent copies of 
$
T\left(B(0,\eta\rho),B(0,\rho-\rho\eta)^c\right).
$
\end{claim}

Let us prove the claim.  
This is a classical application of BK inequality. 
We follow the same plan as the proof of (4.13) in Saint-Flour lecture notes by Kesten \cite{Kesten-saint-flour}. 
Let $b_0,\dots,b_{k-1} \in \R^d$.
For each $j \in \{0,\dots,k-1\}$, 
the travel time from $B(b_{j},\eta\rho)$ to $B(b_{j},\rho-\eta\rho)^c)$ only depends on random balls which touches $B(b_{j},\rho-\eta\rho)$.
This is due to the fact that the infimum wich defines the travel time can be taken over all paths entirely inside $B(b_{j},\rho-\eta\rho)$.
Set $M=\max(\|b_0\|,\dots,\|b_{k-1}\|)+\rho$.
By the previous remark, the event on the left-hand side of the claim only depends on restriction $\xi_U$ of $\xi$ to the set 
\[
U=\{(c,r) \in \R^d \times (0,+\infty) : B(c,r) \text{ touches } B(0,M)\} = \{(c,r) \in \R^d \times (0,+\infty) : \|c\| \le r+M\}.
\]
As $\int r^d \nu(\d r)$ is finite, $\d c \otimes \nu(\d r) (U)$ is finite.

Set
\[
\Lambda = \{(t_0,\dots,t_{k-1}) \in \Q^k : t_0+\dots+t_{k-1} < k\rho (\mu(\nu)-\epsilon)\}.
\]
For all $j \in \{0,\dots,k-1\}$ and all $(t_0,\dots,t_{k-1}) \in \Q^k$ we set
\[
F^t_j = \big\{ \text{finite subsets } S \subset \R^d \times (0,+\infty) : T\big(B(b_j	,\eta\rho),B(b_j,\rho-\eta\rho)^c ; S\big) < t_j\big\}.
\]
Thanks to the remark above and with the notations of Section \ref{s:bk} in Appendix, we can thus write
\begin{align*}
&  \P\left[ \inf \left\{ \sum_{j=0}^{k-1} T\left(B(b_{j},\eta\rho),B(b_{j},\rho-\eta\rho)^c;\xi^j\right) : 
\xi^0, \dots, \xi^{k-1} \text{ are disjoint finite subsets of } \xi \right\} < k\rho (\mu(\nu)-\epsilon)\right] \\
&  \le \P\left[ \inf \left\{ \sum_{j=0}^{k-1} T\left(B(b_{j},\eta\rho),B(b_{j},\rho-\eta\rho)^c;\xi^j\right) : 
\xi^0, \dots, \xi^{k-1} \text{ are disjoint finite subsets of } \xi_U \right\} < k\rho (\mu(\nu)-\epsilon)\right] \\
& \le \P\left[ \bigcup_{(t_0,\dots,t_{k-1}) \in \Lambda} \quad \bigcup_{\xi^0, \dots, \xi^{k-1} \text{ disjoint finite subsets of } \xi_U}  \quad
\bigcap_{j=0}^{k-1} \left\{ T\left(B(b_{j},\eta\rho),B(b_{j},\rho-\eta\rho)^c;\xi^j\right) < t_j \right\}\right] \\
& \le \P\left[\bigcup_{(t_0,\dots,t_{k-1}) \in \Lambda} \{\xi_U \in F^t_0 \circ \cdots \circ F^t_{k-1} \}\right].
\end{align*}
Using Lemma \ref{l:bk}, that is BK inequality (see Section \ref{s:bk} in Appendix), with $p=d+1$, $m=k$, $\varpi$ the restriction of $\d c \otimes \nu(\d r)$ to $U$ we get,
denoting by $\xi^0_U, \dots, \xi^{k-1}_U$ independent copies of $\xi_U$,
\begin{align*}
\P\left[\bigcup_{(t_0,\dots,t_{k-1}) \in \Lambda} \{\xi_U \in F^t_0 \circ \cdots \circ F^t_{k-1} \}\right] 
& \le \P\left[\bigcup_{(t_0,\dots,t_{k-1}) \in \Lambda} \{\xi^0_U \in F^t_0, \dots, \xi^{k-1}_U \in F^t_{k-1} \}\right] \\
& \le  \P\left[\bigcup_{(t_0,\dots,t_{k-1}) \in \Lambda}\bigcap_{j=0}^{k-1} \{T\big(B(b_j	,\eta\rho),B(b_j,\rho-\eta\rho)^c ; \xi^j_U\big) < t_j\}\right] \\
& \le  \P\left[\sum_{j=0}^{k-1} \{T\big(B(b_j	,\eta\rho),B(b_j,\rho-\eta\rho)^c ; \xi^j_U\big) < k\rho (\mu(\nu)-\epsilon)\right].
\end{align*}
The claim now easily follows by stationarity and we resume the proof of Proposition \ref{p:travel-time-skeleton}. 

By Claim \ref{c:bk} and Inequality \eqref{e:ajoutM}
we get, for any $\beta>0$,
\begin{align*}
&\hspace{-2cm} \P\left[ \inf \left\{T^\square(a_0,\dots,a_k) : (a_0,\dots,a_k) \; \rho\text{-skeleton from }0\right\} < k\rho (\mu(\nu)-\epsilon) \right]  \\
&\le  C^{k-1}\P\left[\frac 1 \rho \sum_{j=0}^{k-1} T^{(j)} \le k (\mu(\nu)-\epsilon) \right] \\
&\le  C^{k-1}\E\left[e^{\beta  \sum_{j=0}^{k-1} \left( \mu(\nu)-\epsilon -  \frac {T^{(j)}} {\rho }\right) } \right] \\
&\le  \left(C\E\left[e^{\beta  	\left(\mu(\nu)-\epsilon -  \frac {T^{(0)}} {\rho }\right) } \right]\right)^k \\
&\le  \left(Ce^{\beta(\mu(\nu)-\epsilon)} \P\left[\frac{T^{(0)}}\rho \le \mu(\nu)-\epsilon/2\right]+Ce^{-\beta\epsilon/2}\right)^k \\
&\le  \left(Ce^{\beta} \P\left[\frac{T^{(0)}}\rho \le \mu(\nu)-\epsilon/2\right]+Ce^{-\beta\epsilon/2}\right)^k 
\end{align*}
as $\mu(\nu) \le 1$.
Let $\beta=\beta(\epsilon,d)$ be such that
\[
Ce^{-\beta\epsilon/2} \le \frac 1 4.
\]

\begin{claim}\label{c:commedansm1}
\[
\frac{T^{(0)}}\rho \to \mu(\nu)(1-2\eta)
\]
almost surely as $\rho \to \infty$.
\end{claim}
Let us prove the claim. Any path from $0$ to $S(\rho(1-\eta))$, the sphere of radius $\rho(1-\eta)$ centered at $0$,
can be seen as the concatenation of a first path from $0$ to $S(\rho\eta)$ and a second path from $S(\rho\eta)$ to $S(\rho(1-\eta))$.
Taking infimums, we get
\begin{equation}\label{e:vmc1}
T\big(0,S(\rho\eta)\big) + T\big(S(\rho\eta),S(\rho(1-\eta))\big) \le T\big(0,S(\rho(1-\eta))\big).
\end{equation}
On the other hand, for any $x$ in $S(\rho\eta)$ we have
\begin{eqnarray*}
 T\big(0,S(\rho(1-\eta))\big)
  & \le  & T(0,x) + T\big(x,S(\rho(1-\eta))\big)\\
 & \le & \left(\sup_{x' \in S(\rho\eta)} T(0,x')\right) + T\big(x,S(\rho(1-\eta))\big).
\end{eqnarray*}
Taking the infimum in $x$, we now get
\begin{equation}\label{e:vmc2}
T\big(0,S(\rho(1-\eta))\big)  \le \left(\sup_{x' \in S(\rho\eta)} T(0,x') \right) + T\big(S(\rho\eta),S(\rho(1-\eta))\big).
\end{equation}
From \eqref{e:vmc1} and \eqref{e:vmc2} we get
\[
 T\big(0,S(\rho(1-\eta))\big) -\left(\sup_{x' \in S(\rho\eta)} T(0,x') \right)  \le T\big(S(\rho\eta),S(\rho(1-\eta))\big) 
 \le T\big(0,S(\rho(1-\eta))\big) - T\big(0,S(\rho\eta)\big).
\]
By Theorem \ref{t:timeconstant} we then deduce the claim. 

We resume the proof of Proposition \ref{p:travel-time-skeleton}.
By the previous claim we have 
\[
\frac{T^{(0)}}\rho \to \mu(\nu)(1-2\eta) > \mu(\nu)-\epsilon/2
\]
almost surely as $\rho \to \infty$.
We can then fix $\rho_0=\rho_0(\epsilon,d,\nu)$ such that, for all $\rho \ge \rho_0$,
\[
Ce^{\beta} \P\left[\frac{T^{(0)}}\rho \le \mu(\nu)-\epsilon/2\right] \le \frac 1 4.
\]
For such $\rho$, we thus have
\[
\P\left[ \inf \left\{T^\square(a_0,\dots,a_k) : (a_0,\dots,a_k) \; \rho\text{-skeleton from }0\right\} 
< k\rho (\mu(\nu)-\epsilon) \right]  \le \frac 1 {2^k}.
\]
Therefore, for any $k_0$,
\[
\P\left[\inf \left\{ \frac 1 {k \rho} T^\square(a_0,\dots,a_k) :  k \ge k_0 \text{ and } (a_0,\dots,a_k) \text{ is a }\rho\text{ skeleton from }0\right\}
 < \mu(\nu)-\epsilon\right]\le \frac 1 {2^{k_0-1}}.
\]
By Borel-Cantelli lemma we deduce that, with probability one, for $k_0$ large enough,
\[
\inf \left\{ \frac 1 {k \rho} T^\square(a_0,\dots,a_k) :  k \ge k_0 \text{ and } (a_0,\dots,a_k) \text{ is a }\rho\text{ skeleton from }0\right\} 
\ge \mu(\nu)-\epsilon.
\]
Therefore $\mu^\square(\rho,\nu) \ge \mu(\nu)-\epsilon$ for all $\rho \ge \rho_0$. 
\end{proof}

\subsection{Lower bound on $\tau(\pi)/\ell_\rho(\pi)$}
\label{s:tau-ltaurho}

The aim of Section \ref{s:tau-ltaurho} is to prove the following result.

\begin{prop} \label{p:tau-ltaurho} 
Let $\hat\nu$ be a finite measure on $(0,+\infty)$ which satisfies the greedy condition \eqref{e:greedy}.
Let $\epsilon>0$. 
There exists $\rho_0=\rho_0(\hat\nu, \epsilon,d)$ such that the following holds.
Let $\rho \ge \rho_0$ and let $\nu \preceq \hat\nu$ be a measure on $(0,+\infty)$.
Let $\xi$ be a Poisson point process driven by $\nu$. 
Then, almost surely,
 \[
 \lim_{s \to \infty}\left[ \inf \left\{\frac{\tau(\pi )}{\ell_\rho(\pi)}, \; \pi \text{ is a }\xi\text{-good path from }0\text{ to }B(0,s)^c\right\}\right] \ge \mu(\hat\nu)-\epsilon.
 \]
\end{prop}

\begin{proof} Let $\nu \preceq \hat\nu$ and $\epsilon>0$ be as in the statement of the proposition.
We couple $\xi$, driven by $\nu$, and $\hat\xi$, driven by $\hat\nu$, in the natural way explained in Section \ref{s:couplings}.
Let $\rho >0$ and $s > \rho$.
By Proposition \ref{p:reduction} applied with $\delta=\epsilon$ we get
\begin{align*}
 \inf_{\pi : 0 \to B(0,s)^c} \frac{\widetilde\tau(\pi ; \xi)}{\ell(\pi_\rho)} 
& \ge  \inf \left\{ \frac 1 {k \rho} T^\square\left(a_0,\dots,a_k;\hat{\xi} \right) :  k \ge \lfloor s/\rho \rfloor \text{ and } (a_0,\dots,a_k) \text{ is a }\rho\text{ skeleton from }0\right\}   \\
& \qquad \qquad - (2+2\epsilon^{-1}) G\left(s-\rho, \hat{\xi}^{>\epsilon\rho}\right) - 2\epsilon.
\end{align*}
Therefore, by \eqref{e:musquare} and \eqref{e:ginfty},
\[
\lim_{s \to \infty} \inf_{\pi : 0 \to B(0,s)^c} \frac{\widetilde\tau(\pi,\xi)}{\ell(\pi_\rho)}
\ge \mu^\square(\rho,\hat\nu)   - (2+2\epsilon^{-1}) \, G\left(\infty, \hat{\xi}^{>\epsilon\rho}\right) - 2\epsilon.
\]
By Proposition \ref{p:travel-time-skeleton} 
(as $\hat\nu$ satisfies the greedy condition \eqref{e:greedy}, we know by \eqref{e:greedy2} that the $d$-th moment of $\hat\nu$ is finite), 
\[
\liminf_{\rho\to\infty} \mu^\square(\rho,\hat\nu) \ge \mu(\hat\nu).
\]
By Corollary \ref{c:greedy}, there exists $C=C(d)$ such that
\[
G\left(\infty, \hat{\xi}^{>\epsilon\rho}\right)\le C(d) \int_{(0,+\infty)} \hat\nu\big( [\max(\epsilon\rho,r), +\infty)\big)^{1/d} dr.
\]
As $\hat\nu$ satisfies the greedy condition \eqref{e:greedy}, by dominated convergence, we get 
\[
\lim_{\rho \to \infty} \int_{(0,+\infty)} \hat\nu\big( [\max(\epsilon\rho,r), +\infty)\big)^{1/d} dr = 0.
\]
By the previous remarks, there exists $\rho_0=\rho_0(d,\hat\nu,\epsilon)$ such that, for any $\rho \ge \rho_0$,
\[
\mu^\square(\rho,\hat\nu) \ge \mu(\hat\nu)-\epsilon \text{ and } (2+2\epsilon^{-1})G\left(\infty, \hat{\xi}^{>\delta\rho}\right)\le \epsilon
\]
and therefore
\[
\lim_{s \to \infty} \inf_{\pi : 0 \to B(0,s)^c} \frac{\widetilde\tau(\pi;\xi)}{\ell(\pi_\rho)} \ge \mu(\hat\nu)-4\epsilon.
\]
Using Lemma \ref{l:good} we get, for such $\rho$,
\[
\lim_{s \to \infty} \inf_{\tiny \begin{array}{c} \pi : 0 \to B(0,s)^c, \\ \pi \text{ is a }\xi-\text{good path}\end{array}} \frac{\tau(\pi ; \xi)}{\ell(\pi_\rho)} \ge \mu(\hat\nu)-4\epsilon.
\]
This concludes the proof. 
\end{proof}

\subsection{Proof of Theorem \ref{t:controle-l}}
\label{s:preuve:t:controle-l}

We first state and prove the following result.

\begin{prop} \label{p:controle-ltaurho}
Let $\hat\nu$ be a finite measure on $(0,+\infty)$ which satisfies the greedy condition \eqref{e:greedy}.
Assume $\mu(\hat\nu)>0$.
Let $\epsilon>0$. 
There exists $\rho_0=\rho_0(\hat\nu, \epsilon,d)$ such that the following holds.
Let $\rho \ge \rho_0$ and let $\nu \preceq \hat\nu$ be a measure on $(0,+\infty)$.
Let $\xi$ be a Poisson point process driven by $\nu$. 
Then, almost surely,
\[
\limsup_{s \to \infty} \left[\sup\left\{\frac{\ell(\pi_\rho)}{s}, \; \pi \text{ is a good geodesic from } 0 \text{ to } S(0,s)\right\} \right]\le \frac{\mu(\nu)}{\mu(\hat\nu)}+\epsilon.
\]
\end{prop}

\begin{proof}
 Let $\nu \preceq \hat\nu$ and $\xi$ be as in the statement of Proposition \ref{p:controle-ltaurho}.
 Let $\epsilon>0$.
 Using Proposition \ref{p:tau-ltaurho}, fix $\rho_0=\rho_0(\hat\nu,d,\epsilon)$ such that, for any $\rho \ge \rho_0$, almost surely,
\[
\lim_{s \to \infty}\left[ \inf \left\{\frac{\tau(\pi;\xi)}{\ell_\rho(\pi)}, \; \pi \text{ is a }\xi\text{-good path from }0\text{ to }B(0,s)^c\right\}\right] \ge \frac{\mu(\hat\nu)}{1+\epsilon\mu(\hat\nu)}
\]
and therefore\footnote{Note that any $\xi$-good geodesic from $0$ to $S(0,s)$ is a good path from $0$ to $B(0,s)^c$
and that for any $\xi$-good geodesic from $0$ to $S(0,s)$ we have $T(0,S(0,s);\xi) = \tau(\pi;\xi)$.}
\begin{equation*}
\liminf_{s \to \infty}\left[ \inf \left\{\frac{T(0,S(0,s);\xi)}{\ell_\rho(\pi)}, \; \pi \text{ is a }\xi\text{-good geodesic from }0\text{ to }S(0,s)\right\}\right] \ge\frac{\mu(\hat\nu)}{1+\epsilon\mu(\hat\nu)}
\end{equation*}
and then
\begin{equation}\label{e:tic}
\limsup_{s \to \infty}\left[ \sup \left\{\frac{\ell_\rho(\pi)}{T(0,S(0,s);\xi)}, \; \pi \text{ is a }\xi\text{-good geodesic from }0\text{ to }S(0,s)\right\}\right] \le \frac{1+\epsilon\mu(\hat\nu)}{\mu(\hat\nu)}.
\end{equation}
By Theorem \ref{t:timeconstant}, almost surely,
 \begin{equation} \label{e:tac}
 \lim_{s \to \infty} \frac{T(0,S(0,s) ; \xi)}s = \mu(\nu). 
 \end{equation}
 For any $\rho \ge \rho_0$, by \eqref{e:tic} and \eqref{e:tac}, almost surely,
 \[
 \limsup_{s \to \infty}\left[ \sup \left\{\frac{\ell_\rho(\pi)}{s}, \; \pi \text{ is a }\xi\text{-good-geodesic from }0\text{ to }S(0,s)\right\}\right] 
 \le \mu(\nu)\frac{1+\epsilon\mu(\hat\nu)}{\mu(\hat\nu)}
 \le \frac{\mu(\nu)}{\mu(\hat\nu)} + \epsilon
 \] 
 as $\mu(\nu)\le 1$.
\end{proof}

\begin{proof}[Proof of Theorem \ref{t:controle-ltaurho:simple}]
Theorem \ref{t:controle-ltaurho:simple} is a special case of Proposition \ref{p:controle-ltaurho} when $\nu=\hat\nu$.
\end{proof}

\begin{proof}[Proof of Theorem \ref{t:controle-l}]
Let $\hat\nu$ and $\nu$ be as in the statement of Theorem \ref{t:controle-l}.
Use Proposition \ref{p:controle-ltaurho} with $\epsilon=1$. 
We get $\rho_0=\rho_0(\hat\nu,d)$ such that the following holds.
Almost surely,
\begin{equation}\label{e:rentrons1}
\limsup_{s \to \infty} \left[\sup\left\{\frac{\ell(\pi_{\rho_0})}{s}, \; \pi \text{ is a good geodesic from } 0 \text{ to } S(0,s)\right\} \right]\le \frac{\mu(\nu)}{\mu(\hat\nu)}+1 \le \frac{2\mu(\nu)}{\mu(\hat\nu)} \le \frac{2}{\mu(\hat\nu)}
\end{equation}
where we used $1 \ge \mu(\nu) \ge \mu(\hat\nu)$.
Set $\lambda = \nu[(0,+\infty)]$ and $\hat\lambda = \hat \nu[(0,+\infty)]$.
Let $\chi$ be the projection of $\xi$ on $\R^d$.
This is a Poisson point process of intensity $\lambda$ times the Lebesgue measure.
Note $\lambda \le \hat\lambda$.
Apply Proposition \ref{p:controlellrho}.
We get $C'=C'(d)$ such that 
\[
\lim_{k \to \infty} \left[\sup\left\{ \frac{\ell(\pi)}{\ell(\pi_{\rho_0})}, \; \pi \text{ a } \chi\text{-regular path from }0\text{ such that }
\ell(\pi_{\rho_0}) \ge k\rho_0\right\} \right]\le \max(1,\lambda \rho_0^d) C'. 
\]
As a $\xi$-good geodesic from $0$ to $S(0,s)$ is a $\chi$-regular path from $0$ such that $\ell(\pi_{\rho_0}) \ge s-\rho_0$, 	we get
\begin{equation}\label{e:rentrons2}
\limsup_{s \to \infty} \left[\sup\left\{ \frac{\ell(\pi)}{\ell(\pi_{\rho_0})}, \;  \pi \text{ is a good geodesic from } 
0 \text{ to } S(0,s)\right\} \right]\le \max(1,\lambda \rho_0^d) C'.
\end{equation}
Therefore by \eqref{e:rentrons1} and \eqref{e:rentrons2} we get
\[
\limsup_{s \to \infty} \left[\sup\left\{ \frac{\ell(\pi)}s, \;  \pi \text{ is a good geodesic from } 
0 \text{ to } S(0,s)\right\} \right]\le  \frac{2}{\mu(\hat\nu)}\max(1,\hat\lambda \rho_0^d) C'.
\]
This ends the proof\footnote{The proof actually shows 
\[
\limsup_{s \to \infty} \left[\sup\left\{ \frac{\ell(\pi)}s, \;  \pi \text{ is a good geodesic from } 
0 \text{ to } S(0,s)\right\} \right]\le  C\frac{\mu(\nu)}{\mu(\hat\nu)}\max(1,\lambda \rho_0^d) 
\]
for some constant $C=C(d)$.
}.
\end{proof}

\section{Proofs of the main results.}
\label{s:preuve:main}

\subsection{Proof of Theorem \ref{t:main}}
\label{s:preuve:t:main}

Let $\cR_1,\cR_2$ and $\hat\cR$ be as in the statement of Theorem \ref{t:main}.
Define three new admissible maps by
\[
\cR^-=\min(\cR_1,\cR_2), \; \cR^+=\max(\cR_1,\cR_2) \text{ and } \cR^\Delta=\cR^+-\cR^- =|\cR_2-\cR_1|.
\]
Let $\Xi$ be a Poisson point process on $\R^d \times (0,+\infty)$.
As in Section \ref{s:couplings}, we build a point process $\hat\xi$ using $\Xi$ and $\hat\cR$ and introduce an appropriate random map $\hat r$ such that
\[
\hat\xi = \{(c,\hat r(c)), c \in \R^d \text{ such that } \hat r(c) >0\}.
\]
Note that we have shortened our notations:  we wrote $\hat\xi$ instead of $\xi^{\hat\cR}$ and $\hat r(c)$ instead of $r(c ; \xi^{\hat\cR})$.
In a similar way and with the same kind of notations, we define 
$\xi^1$ using $\cR^1$,  
$\xi^2$ using $\cR^2$,  
$\xi^-$ using $\cR^-$,
$\xi^+$ using $\cR^+$
and $\xi^\Delta$ using $\cR^\Delta$.
In particular,
\begin{align*}
\xi^- & =\{(c,r^-(c)), c \in \R^d \text{ such that } r^-(c)>0\}, \\
\xi^+ & =\{(c,r^+(c)), c \in \R^d \text{ such that } r^+(c)>0\}, \\
\xi^\Delta & =\{(c,r^\Delta(c)), c \in \R^d \text{ such that } r^\Delta(c)>0\}
\end{align*}
and, for all $c \in \R^d$,
\begin{equation}\label{e:rdeltadiff}
r^\Delta(c) = r^+(c)-r^-(c).
\end{equation}

Let $s>0$.
Thanks to Lemma \ref{l:good-geodesic}, we can fix a $\xi^+$-good geodesic $\pi^+=(x_0,\dots,x_n)$ from $0$ to $S(0,s)$.
By Lemma \ref{l:modification-sigma} and \eqref{e:rdeltadiff} we get
\begin{align}
\widetilde\tau(\pi^+;\xi^-) 
& \le \widetilde\tau(\pi^+;\xi^+) + 2 \sum_{i \in \{1,\dots,n-1\}} r^\Delta(x_i) \nonumber \\
& \le \widetilde\tau(\pi^+;\xi^+) + 2r(\pi^+ ; \xi^\Delta), \label{e:curry1}
\end{align}
using greedy paths notations.
Lemma \ref{l:good} then yields, as $\pi^+$ is a $\xi^+$-good geodesic,
\begin{equation} \label{e:curry2}
\widetilde\tau(\pi^+ ; \xi^+) = \tau(\pi^+ ; \xi^+) = T(0,S(0,s) ; \xi^+) 
\text{ and } 
T(0,S(0,s);\xi^-) \le \tau(\pi^+ ; \xi^-) \le \widetilde\tau(\pi^+ ; \xi^-).
\end{equation}
Combining \eqref{e:curry1} and \eqref{e:curry2}, we get
\[
T(0,S(0,s) ; \xi^-) \le T(0,S(0,s) ; \xi^+) + 2r(\pi^+ ; \xi^\Delta).
\]
Therefore
\begin{align}
\frac{T(0,S(0,s);\xi^-)}s 
& \le \frac{T(0,S(0,s) ; \xi^+)}s + 2 \frac{r(\pi^+; \xi^\Delta)}{\ell(\pi^+)} \frac{\ell(\pi^+)}s \nonumber \\
& \le  \frac{T(0,S(0,s) ; \xi^+)}s + 2G\big(s;\xi^\Delta)\frac{\ell(\pi^+)}s. \label{e:mer1}
\end{align}

Let $C=C(d,\hat R)$ be the constant given by Theorem \ref{t:controle-l}.
We then have
\begin{equation} \label{e:mer2}
\limsup_{s \to \infty} \left[\sup\left\{ \frac {\ell(\pi)}s, \pi \text{ is a } \xi^+\text{-good geodesic from }0\text{ to }S(0,s)\right\}\right] \le C.
\end{equation}
Let $C'(d)$ be the constant given by Corollary \ref{c:greedy}.
Then
\begin{equation} \label{e:mer3}
G(\infty ; \xi^{\Delta}) \le C'\int_0^\infty \nu^{|\cR_2-\cR_1|}([r,+\infty))^{1/d} \d r.
\end{equation}
By Theorem \ref{t:timeconstant}, \eqref{e:mer1}, \eqref{e:mer2} and \eqref{e:mer3} we get
\begin{align*}
\mu(\nu^{\cR^-}) 
& \le \mu(\nu^{\cR^+}) + 2CC' \int_0^\infty \nu^{|\cR_2-\cR_1|}([r,+\infty))^{1/d} \d r.
\end{align*}
But $\cR^- \le \cR_1 \le \cR^+$ and $\cR^- \le \cR_2 \le \cR^+$.
Therefore $\mu(\nu^{\cR^+})  \le \mu(\nu^{\cR_1}) \le  \mu(\nu^{\cR^-})$ and $\mu(\nu^{\cR^+})  \le \mu(\nu^{\cR_2}) \le  \mu(\nu^{\cR^-})$.
Combined with the previous result, this yields
\[
|\mu(\nu^{\cR_1}) - \mu(\nu^{\cR_2})| \le 2CC' \int_0^\infty \nu^{|\cR_2-\cR_1|}([r,+\infty))^{1/d} \d r.
\]
This concludes the proof. \qed

\subsection{Proof of Corollary \ref{c:main}}
\label{s:preuve:c:main}

We first prove a lemma relating the moment $(d+\eta)$ of a map $\cR$ with the integral appearing in the greedy condition \eqref{e:greedy}.

\begin{lemma} \label{l:quantitegreedy}
Let $\eta>0$ and $\lambda \ge 0$. 
Let $\cR$ be an admissible map from $(0,+\infty)$ to $[0,+\infty)$ such that $\nu^{\cR}[(0,+\infty)] \le \lambda$.
There exists $C=C(d,\lambda,\eta)$ such that
\[
\int_0^\infty \nu^{\cR}([r,+\infty))^{1/d} \d r  \le C \left(\int_{(0,+\infty)} \cR(u)^{d+\eta} \d u\right)^{1/(d+\eta)}.
	\]
\end{lemma}
\begin{proof}
For any $r>0$,
\[
\nu^{\cR}([r,+\infty)) = \int_{(0,+\infty)} \1_{[r,+\infty[}(\cR(u)) \d u \le \int_{(0,+\infty)} \cR(u)^{d+\eta} r^{-(d+\eta)} \d u.
\]
Moreover $\nu^{\cR}([r,+\infty)) \le \nu^{\cR}[(0,+\infty)]$.
Therefore
\begin{align*}
\int_0^\infty \nu^{\cR}([r,+\infty))^{1/d} \d r 
 & \le \int_0^\infty \min\left(\nu^{\cR}[(0,+\infty)], \left( \int_{(0,+\infty)} \cR(u)^{d+\eta} \d u \right) \; r^{-(d+\eta)} \right)^{1/d} \d r \\
 & = \left(\int_{(0,+\infty)} \cR(u)^{d+\eta} \d u\right)^{1/(d+\eta)} \int_0^\infty \min\left(\nu^{\cR}[(0,+\infty)], s^{-(d+\eta)} \right)^{1/d} \d s \\
 & \le \left(\int_{(0,+\infty)} \cR(u)^{d+\eta} \d u\right)^{1/(d+\eta)} \int_0^\infty \min\left(\lambda, s^{-(d+\eta)} \right)^{1/d} \d s.
\end{align*}
The result follows with the choice
$
C=\int_0^\infty \min\left(\lambda, s^{-(d+\eta)} \right)^{1/d} \d s.
$
\end{proof}

\begin{proof}[Proof of Corollary \ref{c:main}]
Set $\lambda = \nu^{\hat\cR}[(0,+\infty)]$.
Let $C=C(d,\lambda,\eta)$ be given by Lemma \ref{l:quantitegreedy}.
As $\cR_1 \le \hat\cR$ and $\cR_2 \le \hat\cR$ we get $\nu^{|\cR_1-\cR_2|}[(0,+\infty)] \le \lambda$.
We can therefore apply Lemma \ref{l:quantitegreedy} and we get
\[
\int_0^\infty \nu^{|\cR_1-\cR_2|}([r,+\infty))^{1/d} \d r  \le C \left(\int |\cR_1(u)-\cR_2(u)|^{d+\eta} \d u\right)^{1/(d+\eta)} 
\]
and
\[
\int_0^\infty \nu^{\hat\cR}([r,+\infty))^{1/d} \d r  \le C \left(\int \hat\cR(u)^{d+\eta} \d u\right)^{1/(d+\eta)} < \infty.
\]

We can then apply Theorem \ref{t:main} and conclude the proof.
\end{proof}

\subsection{Proof of Theorem \ref{t:mainsuite}}
\label{s:preuve:t:mainsuite}

For all $n \ge 0$, define the admissible maps
\[
\cR_n^-=\inf_{k \in \N \cup \{\infty\} : k \ge n}\cR_k \text{ and } \cR_n^+=\sup_{k \in \N \cup \{\infty\} : k \ge n} \cR_k.
\]
Let us note a couple of facts.
\begin{enumerate}
\item For all $n$, $\cR_n^- \le \hat\cR$ and $\cR_n^+ \le \hat\cR$.
\item $\cR_n^-$ and $\cR_n^+$ converge almost surely to $\cR_\infty$.
\end{enumerate}

Let $\Xi$ be a Poisson point process on $\R^d \times (0,+\infty)$.
As in Section \ref{s:couplings}, we build a point process $\hat\xi$ using $\Xi$ and $\hat\cR$ and introduce an appropriate random map $\hat r$ such that
\[
\hat\xi = \{(c,\hat r(c)), c \in \R^d \text{ such that } \hat r(c) >0\}.
\]
As in the proof of Theorem \ref{t:main}, we have shortened our notations.
In a similar way and with the same kind of notations, we define 
$\xi_\infty$ with $\cR_\infty$ and, for all $n$, 
$\xi_n$ with $\cR_n$, 
$\xi_n^-$ with $\cR^-_n$
and $\xi_n^+$ with $\cR^+_n$.
We denote by $\mu_\infty, \mu_n, \mu_n^-$ and $\mu_n^+$ the associated time constants.
We will use similar conventions for $r$, $\nu$, $\tau$ and $T$.

\begin{claim} \label{claim:cvrn-} On a probability one event,
\[	
\forall c \in \R^d, \; \lim_n r_n^-(c)=r_\infty(c).
\]
\end{claim}
\begin{proof}[Proof of Claim \ref{claim:cvrn-}]
Let $A \subset (0,+\infty)$ be a Borel set with zero Lebesgue measure such that 
$\cR_n(u)$ converges to $\cR_\infty(u)$ for all $u \in (0,+\infty) \setminus A$.
The Lebesgue measure of $\R^d \times A$ is zero.
As a consequence, on a probability one event, there is no points of $\Xi$ in $\R^d \times A$.
On this full probability event, we then have 
\[
\forall (c,u) \in \Xi, \; \lim_n \cR_n(u) = \cR_\infty(u)
\]
and therefore
\[
\forall (c,u) \in \Xi, \; \lim_n \cR^-_n(u) = \cR_\infty(u)
\]
and thus we get (distinguishing whether there exists $u \in (0,+\infty)$ such that $(c,u) \in \Xi$ or not) 
\[
\forall c \in \R^d, \; \lim_n r^-_n(c) = r_\infty(c).
\]
The claim is proven.
\end{proof}

For all $n$,
\[
\mu_n^+ \le \mu_n \le \mu_n^-.
\]
To conclude the proof, it is therefore sufficient to prove
\begin{equation}\label{e:mun+}
\liminf_n \mu_n^+ \ge \mu_\infty
\end{equation}
and
\begin{equation}\label{e:mun-}
\limsup_n \mu_n^- \le \mu_\infty.
\end{equation}

Let us first prove \eqref{e:mun-}.
This is a simple consequence of the following claim.
\begin{claim} \label{claim:pasmalmat2} For all $k \ge 1$,
\[
\lim_n \E\left[ \frac{T_n^-(0,ke_1)}{k}\right] = \E\left[\frac{T_\infty(0,ke_1)}{k}\right].
\]
\end{claim}
Let us first prove \eqref{e:mun-} assuming Claim \ref{claim:pasmalmat2}.
Let $\epsilon>0$. 
By Theorem \ref{t:timeconstant} we can fix $k$ be such that 
\[
\frac{\E[T_\infty(0,ke_1)]}{k} \le \mu_\infty + \epsilon.
\]
By Claim \ref{claim:pasmalmat2} we can fix $n_0$ such that, for all $n \ge n_0$,
\[
\E\left[ \frac{T_n^-(0,ke_1)}{k}\right] \le \E\left[\frac{T_\infty(0,ke_1)}{k}\right] + \epsilon \le \mu_\infty+2\epsilon.
\]
But for any $n$, by \eqref{e:nuinfimum} we know that $\mu_n^-$ is defined as an infimum thus
\[
\mu_n^- \le \E\left[ \frac{T_n^-(0,ke_1)}{k}\right]
\]
and therefore, for any $n \ge n_0$,
\[
\mu_n^- \le \mu_\infty+2\epsilon.
\]
Thus, \eqref{e:mun-} is proven assuming Claim \ref{claim:pasmalmat2}.
We now prove this claim.

\begin{proof}[Proof of Claim \ref{claim:pasmalmat2}]
Let $k \ge 1$. 
We first prove the following almost sure equality:
\begin{equation} \label{e:pasmalmat}
\lim_n \frac{T_n^-(0,ke_1)}{k} = \frac{T_\infty(0,ke_1)}k.
\end{equation}
Let $\epsilon>0$.
Let $\pi'$ be a path from $0$ to $ke_1$ such that $\tau_\infty(\pi') \le T_\infty(0,ke_1)+\epsilon$.
By Lemma \ref{l:good-geodesic-pi} we can fix a path $\pi=(x_0,\dots,x_m)$ from $0$ to $ke_1$ 
such that $\widetilde\tau_\infty(\pi) \le \tau_\infty(\pi')$.
We then have
\begin{align*}
T_\infty(0,ke_1) + \epsilon 
& \ge \widetilde\tau_\infty(\pi) \\
& \ge \widetilde\tau_n^-(\pi) - 2 \sum_{i \in \{1,\dots,m-1\}} \big(r_\infty(x_i)-r_n^-(x_i)\big) \text{ by Lemma } \ref{l:modification-sigma}  \\
& \ge \tau_n^-(\pi) - 2 \sum_{i \in \{1,\dots,m-1\}} \big(r_\infty(x_i)-r_n^-(x_i)\big) \text{ by Lemma } \ref{l:good}.
\end{align*}
But, by Claim \ref{claim:cvrn-}, on a probability one event, for any such path $\pi$, we have
\[
\lim_n \sum_{i \in \{1,\dots,n-1\}} \big(r_\infty(x_i)-r_n^-(x_i)\big) = 0 
\]
and therefore
\[
\limsup_n \tau_n^-(\pi) \le T_\infty(0,ke_1) +\epsilon
\]
and then $\limsup_n T_n^-(0,ke_1) \le  T_\infty(0,ke_1)+\epsilon$. So,

\[
\limsup_n T_n^-(0,ke_1) \le  T_\infty(0,ke_1).
\]
But for all $n$, $\cR_n^- \le \cR_\infty$ and therefore $T_\infty(0,ke_1) \le T_n^-(0,ke_1)$.
This yields \eqref{e:pasmalmat}.
For all $n$, $0 \le T_n^-(0,ke_1) \le k$.
By \eqref{e:pasmalmat} and by the dominated convergence theorem, we finally conclude the proof.
\end{proof}

At this stage of the proof, \eqref{e:mun-} is proven. 
We now prove \eqref{e:mun+}. 
This will end the proof.
If $\mu_\infty=0$, the result is trivial.
Henceforth, we assume $\mu_\infty>0$.
First note that, if $\cR,\cS$ are two admissible maps such that $\cR \le \cS$, then for all $r>0$,
\[
\nu^\cR([r,+\infty))
 = \int_0^{+\infty} \1_{[r,+\infty)}[ \cR(u)] \d u
 \le \int_0^{+\infty} \1_{[r,+\infty)}[\cS(u)] \d u
 = \nu^\cS([r,+\infty)).
\]
From Item 2 of Theorem \ref{t:mainsuite} we thus get that $\nu_\infty$ and all the $\nu_n^+$ satisfy the greedy condition \eqref{e:greedy}.
As $\mu_\infty>0$ and as $\nu_\infty$ satisfies \eqref{e:greedy} we get, by Theorem \ref{t:gt17}, 
that $\nu_\infty$ is strongly subcritical for percolation (see \eqref{e:nupositif} for the definition of strongly subcritical).
By Lemma \ref{l:subopen} (see Appendix \ref{s:open} below) applied to the admissible maps $\hat\cR, \cR_\infty$ and the $\cR_n^+$ (recall that the greedy condition \eqref{e:greedy} implies condition \eqref{e:momentd}, see \eqref{e:greedy2}), 
we get the existence of $n_0$ such that $\nu_{n_0}^+$ is strongly subcritical for percolation.
As $\nu_{n_0}^+$ satisfies \eqref{e:greedy}, we deduce back by Theorem \ref{t:gt17} that $\mu_{n_0}^+$ is positive.
For all $n \ge n_0$, $\cR_n^+ \le \cR_{n_0}^+$ 
and we can then apply Theorem \ref{t:main} to the admissible maps $\cR_n^+$, $\cR_\infty$ and $\cR^+_{n_0}$.
We get
\begin{equation}\label{e:enfin}
\left|\mu_n^+-\mu_\infty\right| \le C \left(\int_0^\infty \left[\nu^{\cR^+_n-\cR_\infty}([r,+\infty))\right]^{1/d} \d r\right)
\end{equation}
where $C=C(d,\cR_{n_0}^+)$ is provided by Theorem \ref{t:main}.
For all $n \ge n_0$, $\cR_n^+-\cR_\infty$ is non-negative and, for all $r>0$ and all $u>0$,
\begin{equation}\label{e:dom1}
\1_{[r,+\infty)} \big[(\cR_n^+-\cR_\infty)(u)\big]  \le \1_{[r,+\infty)} \big[\cR_{n_0}^+(u)\big].
\end{equation}
Recall
\begin{equation}\label{e:dom2}
\int_0^\infty \1_{[r,+\infty)}\big[ \cR_{n_0}^+(u) \big]\d u = \nu_{n_0}^+([r,+\infty)) < \infty.
\end{equation}
Moreover, for almost all $u$, $\cR_n^+(u)$ converges to $\cR_\infty(u)$.
Therefore,  for almost all $u$,
\begin{equation}\label{e:dom3}
\lim_{n\to\infty} \1_{[r,+\infty)} \big[(\cR_n^+-\cR_\infty)(u)\big] = 0 \text{ for all } r>0.
\end{equation}
By \eqref{e:dom1}, \eqref{e:dom2}, \eqref{e:dom3} we get, using the dominated convergence theorem, for all $r>0$,
\[
\lim_{n \to \infty} \int_0^\infty \1_{[r,+\infty)} \big[(\cR_n^+-\cR_\infty)(u)\big] \d u = 0
\]
that is
\begin{equation}\label{e:dom4}
\lim_{n \to \infty} \nu^{\cR^+_n-\cR_\infty}([r,+\infty)) = 0.
\end{equation}
But, for all $n \ge n_0$ and all $r>0$,  after integrating \eqref{e:dom1} we get
\begin{equation}\label{e:dom5}
\nu^{\cR^+_n-\cR_\infty}([r,+\infty)) \le \nu_{n_0}^+([r,+\infty)).
\end{equation}
As $\nu_{n_0}^+$ satisfies \eqref{e:greedy}, by \eqref{e:dom4} and \eqref{e:dom5} we get, using again the dominated convergence theorem, 
\[
\lim_{n \to \infty} \int_0^\infty \left[\nu^{\cR^+_n-\cR_\infty}([r,+\infty))\right]^{1/d} \d r = 0.
\]
Using \eqref{e:enfin} we then get \eqref{e:mun+} which ends the proof.

\section{Proof of Corollary \ref{c:mainsuitesuite}}

\begin{proof}[Proof of Corollary \ref{c:mainsuitesuite}]
Let $\hat \lambda = \sup \{ \lambda_n , n\in \N \cup \{\infty\} \}$. We have $\hat \lambda <\infty$ since $(\lambda_n)_n$ is convergent. Let $\hat\cR$ be the admissible map associated with $\hat \lambda \P_{\hat R}$ by \eqref{e:skorokhod}.
Note that, with the notations of Section \ref{s:couplings}, $\nu^{\hat\cR} = \hat \lambda \P_{\hat R}$ (see below \eqref{e:skorokhod}).
For all $n \in \N \cup \{\infty\}$ let $\cR_n$ be the admissible map associated with $\lambda_n \P_{R_n}$ by \eqref{e:skorokhod}.
As above, for all $n \in \N \cup \{\infty\}$, $\nu^{\cR_n} = \lambda \P_{R_n}$.
We now check that Theorem \ref{t:mainsuite} applies.
Item 1 of Theorem \ref{t:mainsuite} is a consequence of Item 1 of Corollary \ref{c:mainsuitesuite} and the upper bound $\lambda_n \leq \hat \lambda$ for all $n \in \N \cup \{\infty\}$ (see the paragraph about domination and coupling in Section \ref{s:couplings}).
Item 2 of Theorem \ref{t:mainsuite} is a consequence of Item 2 of Corollary \ref{c:mainsuitesuite}.
Item 3 of Theorem \ref{t:mainsuite} is a consequence of Items 3 and 4 of Corollary \ref{c:mainsuitesuite}
and of a variant of a classical result for the representation of a measure by an admissible map\footnote{Let us give a proof. 
For each $u>0$ write $I(u) = \{r>0 : \lambda_\infty \P[R_\infty \ge r]=u\}$. 
The sets $I(u)$ are pairwise disjoint intervals. Therefore
\[
A = \{u>0 : I(u) \text{ is an interval of positive length}\}
\]
is at most countable. To conclude the proof, it suffices to show that, for any $u \in (0,+\infty) \setminus A$,
\[
\limsup \cR_n(u) \le \cR_\infty(u) \text{ and } \liminf \cR_n(u) \ge \cR_\infty(u).
\]
Let $u \in (0,+\infty) \setminus A$. 
For any $\epsilon>0$ there exists $r \in (\cR_\infty(u), \cR_\infty(u)+\epsilon)$ such that
$r$ is a continuity point of $r \mapsto \P[R_\infty \ge r]$.
For any such $r$, $\lim_n \P[R_n \ge r] = \P [R_\infty \ge r]$ since $\P^{R_n}$ converges weakly to $\P^{R_\infty}$. Since $\lim_n \lambda_n = \lambda_{\infty}$ we get $\lim_n \lambda_n \P[R_n \ge r] = \lambda_{\infty} \P[R_\infty \ge r] < u$ and therefore,
if $n$ is large enough, $\lambda_n \P[R_n \ge r]<u$ and therefore $\cR_n (u) \le r < \cR_\infty(u)+\epsilon$. 
This proves the result about the $\limsup$.
If $\cR_\infty(u)=0$ the result about the $\liminf$ is trivial.
Assume henceforth $\cR_\infty(u)>0$.
For any $\epsilon>0$ there exists $r \in (\cR_\infty(u)-\epsilon, \cR_\infty(u))$ such that
$r$ is a continuity point of $r \mapsto \P[R_\infty \ge r]$ and $\lambda_\infty \P[R_\infty \ge r] \neq u$
(because $u \not \in A$) and then $\lambda_\infty \P[R_\infty \ge r]>u$.
For any such $r$, $\lim_n \lambda_n \P[R_n \ge r] = \lambda_\infty \P[R_\infty \ge r] > u$ and therefore,
if $n$ is large enough, $\lambda_n \P[R_n \ge r]>u$ and therefore $\cR_n (u) \ge r \ge \cR_\infty(u)-\epsilon$. 
This proves the result about the $\liminf$.}.
We can now apply Theorem \ref{t:mainsuite} to conclude.
\end{proof}

\appendix

\section{Some remarks on measurability}
\label{s:mesurabilite}

\paragraph{Space of configurations.}
Let $(\cS,\cA)$ denote the usual set of configurations for point processes in $\R^d \times (0,+\infty)$.
In other words, $\cS$ is the set of locally finite subsets of $\R^d \times (0,+\infty)$ and $\cA$ is the $\sigma$-algebra generated
by the family of all maps $S \mapsto \#(S \cap B)$ where $B$ is a Borel subset of $\R^d \times (0,+\infty)$.

There exists a sequence $(X_n)_n$ of measurable maps from $\cS$ to $\R^d \times (0,+\infty)$ and a measurable map $N$ from $\cS$ to 
$\overline \N = \N \cup \{\infty\}$ such that, for all $S \in \cS$,
\[
S = \{X_i(S), 1 \le i \le N(S)\}
\]
and the $X_i(S), 1 \le i \le N(S)$ are pairwise distinct. We write $X_i (S) = (Y_i (S), R_i (S))$ where $Y_i$ (resp. $R_i$) takes values in $\R^d$ (resp. $(0,+\infty)$).

\paragraph{The map from $\R^d \times \cS$ to $\R$ defined by $(z,S) \mapsto \1_{\Sigma(S)}(z)$ is measurable where $\Sigma(S) = \bigcup_{(c,r) \in S} B(c,r)$.}
Indeed,
\[
\{z \in \Sigma(S)\} = \bigcup_{1 \le i \le N(S)} \{\|z-Y_i(S)\| < R_i(S)\}.
\]

\paragraph{Regularity of the map from $\R^d \times \R^d \times \cS$ to $\R$ defined by $(x,y,S) \mapsto \tau(x,y ; S)$.}
Let $x,y \in \R^d$. The map $S \mapsto \tau(x,y;S)$ is measurable.
If $x=y$ this is straightforward.
If $x \neq y$, setting $u = \|y-x\|^{-1}(y-x)$ we can write
\[
\tau(x,y ; S) = \int_0^1 \|y-x\|\1_{\Sigma(S)}(x+ \|y-x\| t u) \d t.
\]
As the map $(z,S) \mapsto \1_{\Sigma(S)}(z)$ is measurable, the above map is measurable by Fubini theorem.

Let $S \in \cS$. The map $(x,y) \mapsto \tau(x,y;S)$ is continuous.
Let $(x_n)_n$ and $(y_n)_n$ be two sequences of $\R^d$ converging to $x$ and $y$.
It suffices to show that $\tau(x_n,y_n;S)$ converges to $\tau(x,y;S)$.
If $x=y$ this is straighforward as, for any $u,v \in \R^d$, $\tau(u,v ;S) \le \|v-u\|$.
If $x \neq y$ we can write, for any $n$ large enough and with $u_n = \|y_n-x_n\|^{-1}(y_n-x_n)$,
\[
\tau(x_n,y_n ; S) = \int_0^1 \|y_n-x_n\|\1_{\Sigma(S)}(x_n+ \|y_n-x_n\| t u_n) \d t.
\]
The set $M=\{t \in [0,1], x+\|y-x\|tu \in \partial\Sigma(S)\}$ is at most countable as the boundary of $\Sigma(S)$ is contained in the union of 
an at most countable number of spheres.
But $\|y_n-x_n\|\1_{\Sigma(S)}(x_n+ \|y_n-x_n\| t u_n)$ converges to $\|y-x\|\1_{\Sigma(S)}(x+ \|y-x\| t u)$ when $n$ tends to $\infty$ for
any $t \in [0,1] \setminus M$.
The required continuity result then follows from the dominated convergence theorem.

\paragraph{Measurability of the map from $\cS$ to $\R$ defined by $S \mapsto T(a,b ; S)$ for $a,b \in \R^d$ given.}
By continuity of $(x,y) \mapsto \tau(x,y ; S)$ for a given $S$, we can write
\[
T(a,b ; S) = \inf_\pi \tau(\pi ; S)
\]
where the infimum is over all paths from $a$ to $b$ whose points (except maybe $a$ or $b$) have rational coordinates.
The result then follows from the measurability of $S \mapsto \tau(x,y;S)$ for given $x,y$.

\paragraph{Measurability of the map from $\cS$ to $\R$ defined by $S \mapsto T(A,B ; S)$ for $A,B \subset \R^d$ given.}
We will only consider very regular sets $A,B\subset \mathbb{R}^d$ such as balls or spheres. By the triangular inequality we get $|T(a,b ; S) - T(a',b ; S')| \leq T(a,a' ; S) + T(b,b' ; S) \leq \| a-a'\| + \| b-b'\|$, thus the infimum in the definition of $T(A,B ; S)$,
\[
T(A,B;S) = \inf \{  T(a,b;S) : a\in A, b\in B \} 
\]
can be taken along any set of points $A'\subset A$ (resp. $B'\subset B$) that is dense in $A$ (resp. in $B$) for the Euclidean norm. If the sets $A$ and $B$ admit dense subsets that are countable (as it is the case for balls or spheres), then the measurability of $S \mapsto T(A,B ; S)$ follows.

\paragraph{Measurability of the map from $\cS$ to $\R$ defined by $S \mapsto T^\square(a_0,\dots,a_k ; S)$ for $a_0,\dots,a_k$ given.}
We can rewrite
\[
T^\square(a_0,\dots,a_k ; S) = \inf \left\{ \sum_{j=1}^k T(a_{j-1},a_j;	\{X_i(S), i \in  I_j\}) :
I_1, \dots, I_k \text{ are disjoint finite subsets of } \{1,\dots,N(S)\}\right\}.
\]
The result then follows from the previous paragraph (and the fact that for any $I \subset \N$, $S \mapsto \{X_i(S), i \in I\}$ is a measurable
map from $\cS$ to $\cS$ as, for all Borel subset $B$, $\#(B \cap \{X_i(S), i \in I\}) = \sum_{i \in I} \1_{B}(X_i(S))$ is measurable.

\paragraph{Measurability of the map from $\cS$ to $\R$ defined by }
\[ S \mapsto \inf \left\{ \frac{1}{k\rho}T^\square(a_0,\dots,a_k ; S) : k\geq k_0  \text{ and } (a_0,\dots,a_k) 
\text{ is a }\rho\text{-skeleton from }0 \right\}.
\]

Let $(a_0,\dots,a_k), (a'_0,\dots,a'_k) \in (\R^d)^{k+1}$.
For all disjoint finite subsets $I_1, \dots, I_k$ of $\{1,\dots,N(S)\}$, for all $j\in \{1,\dots, k\}$,
%$\xi^1,\dots,\xi^k$ of $\xi$ and all $j \in \{1,\dots,k\}$,
\begin{align*}
T(a'_{j-1},a'_j & ;\{X_i(S), i \in  I_j\}) \\
& \le T(a'_{j-1},a_{j-1};\{X_i(S), i \in  I_j\}) + T(a_{j-1},a_j;\{X_i(S), i \in  I_j\})  + T(a_j,a'_j;\{X_i(S), i \in  I_j\}) \\
& \le T(a_{j-1},a_j;\{X_i(S), i \in  I_j\}) + \|a'_{j-1}-a_{j-1}\|+\|a'_j-a_j\|.
\end{align*}
Summing on $j$ and taking the infimum over the disjoint finite subsets $I_1, \dots, I_k$ of $\{1,\dots,N(S)\}$ we get
\[
T^\square(a'_0,\dots,a'_k ; S) \le T^\square(a_0,\dots,a_k ; S) + 2 \sum_j \|a_j-a'_j\|.
\]
By symmetry, we get the reverse inequality and thus
\[
|T^\square(a_0,\dots,a_k ; S) - T^\square(a'_0,\dots,a'_k ; S)| \le 2 \sum_j \|a_j-a'_j\|.
\]
Thanks to this Lipschitz property, the quantity we are interested can be written as an infimum of measurable maps
over a countable set of skeletons.

\section{Openness of the strongly subcritical regime for percolation}
\label{s:open}

The definition of strongly subcritical is given en \eqref{e:nupositif}. It applies to a finite measure $\nu$ on $(0,+\infty)$, but via the correspondance between such measures and admissible maps (see Section \ref{s:couplings}), it also applies to admissible maps.
The aim of this section is to give a proof of the following result.
The result is essentially not new.

\begin{lemma}  \label{l:subopen}
Let $\hat\cR$ and $\cR_\infty$ be two admissible maps. 
Let $(\cR_n)_n$ be a sequence of admissible maps.
Assume the following.
\begin{enumerate}
\item $\cR_\infty \le \hat\cR$ and, for all $n$, $\cR_n \le \hat\cR$.
\item $\int_{(0,+\infty)} \hat\cR(u)^d \d u$ is finite.
\item $\cR_\infty$ is strongly subcritical for percolation.
\item $\cR_n$ converges almost everywhere to $\cR$.
\end{enumerate}
Then there exists $n_0$ such that, for all $n \ge n_0$, $\cR_n$ is strongly subcritical for percolation.
\end{lemma}

The following result is a rephrasing of the core of Proposition A.2 and Lemma A.3 from \cite{Gouere-Theret-17} with the notations of the current paper.

\begin{lemma} \label{l:lemme-ouvert-GT17}
There exists $K=K(d)$ such that the following holds.
Let $\cR$ be an admissible map satisfying
\begin{equation}\label{e:map-moment}
\int_0^{+\infty} \cR(u)^d \d u < \infty.
\end{equation}
Let $\nu^\cR$ be the associated measure and let $\Sigma=\Sigma(d,\nu^\cR)$ be a Boolean model driven by $\nu^\cR$.
Let $M>0$.
Assume
\begin{equation}\label{e:charme}
K^2 \int \cR^d(u) \1_{[M,+\infty)}(\cR(u)) \d u \le \frac 1 4
\end{equation}
and
\begin{equation}\label{e:chene}
K \P[\text{There exists a path inside } \Sigma \text{ from } B(0,M) \text{ to } B(0,2M)^c\big] \le \frac 1 2.
\end{equation}
Then 
\begin{equation}\label{e:acacia}
\P[\text{There exists a path inside } \Sigma\text{ from } B(0,\alpha) \text{ to } B(0,2\alpha)^c\big] \to 0 \text{ as } \alpha \to \infty.
\end{equation}
In other words, $\nu^\cR$ is strongly subcritical for percolation.
\end{lemma}

\begin{proof} 
Let us start by comparing the notations and the settings of \cite{Gouere-Theret-17} and Lemma \ref{l:lemme-ouvert-GT17}.
In \cite{Gouere-Theret-17} the Boolean model is driven by a measure denoted by $\lambda \nu$ 
where $\lambda \in (0,+\infty)$ and $\nu$ is a probability measure on $(0,+\infty)$ whose $d$-moment is finite.
In Lemma \ref{l:lemme-ouvert-GT17}, the Boolean  model is driven by $\nu^{\cR}$
where $\cR$ is an admissible map satisfying \eqref{e:map-moment} and thus $\nu^{\cR}$ is a finite measure on $(0,+\infty)$ whose $d$-moment is finite. 
Therefore, whenever $\nu^{\cR}[(0,+\infty)]>0$ we are in the framework of \cite{Gouere-Theret-17} by setting 
$\lambda=\nu^{\cR}[(0,+\infty)]$ and $\nu = \lambda^{-1} \nu^{\cR}$.
When, on the contrary, $\nu^{\cR}[(0,+\infty)]=0$, everything is trivial and we will not mention further this case.

Let $K_1=K_1(d)$ be given by Proposition A.2 of \cite{Gouere-Theret-17}.
Let $K_2=K_2(d)$ be such that, 
for all $\beta \in [10,100]$, 
there exists a set $A \subset S(0,\beta)$ of cardinality at most $K_2(d)$ such that $S(0,\beta) \subset \cup_{a \in A} B(a,1)$.
Set $K=K(d)=K_1K_2$.

Let $\cR$ be an admissible map satisfying \eqref{e:map-moment}.
Let $M>0$.
Assume that \eqref{e:charme} and \eqref{e:chene} hold.
We have to prove that \eqref{e:acacia} holds.

For all $\alpha \in [10M,100M]$, we have
\begin{equation}\label{e:chene-2}
K_1 \P[\text{There exists a path inside } \Sigma \text{ from } B(0,\alpha) \text{ to } B(0,2\alpha)^c\big] \le \frac 1 2.
\end{equation}
This is a simple and classical consequence of \eqref{e:chene}.
Set $\beta = \alpha M^{-1} \in [10,100]$.
There exists $A \subset S(0,\beta)$ of cardinality at most $K_2(d)$ such that $S(0,\beta) \subset \cup_{a \in A} B(a,1)$.
Multiplying by $M$ and using $M\beta = \alpha$ we thus have 
\[
S(0,\alpha) \subset \cup_{a \in A} B(Ma, M).
\]
Moreover, for all $a \in A$,
\[
B(Ma,2M) \subset B(0,2\alpha).
\]
Indeed $\|Ma\|=\alpha$ and $2M < 10M \le \alpha$.
Therefore, if there exists a path in $\Sigma$ from $B(0,\alpha)$ to $B(0,2\alpha)^c$, then there exists $a \in A$ and a path in $\Sigma$ from
$B(Ma,M)$ to $B(Ma,2M)^c$. By a union bound and by stationarity, we thus get
\begin{align*}
& K_1 \P[\text{There exists a path inside } \Sigma\text{ from } B(0,\alpha) \text{ to } B(0,2\alpha)^c\big]  \\
& \le K_1K_2 \P[\text{There exists a path inside } \Sigma\text{ from } B(0,M) \text{ to } B(0,2M)^c\big] \\
& \le \frac 1 2
\end{align*}
by \eqref{e:chene} as $K=K_1K_2$. Therefore \eqref{e:chene-2} holds.
Applying (A.2) of \cite{Gouere-Theret-17} we get 
\begin{equation}\label{e:chene-3}
\forall \alpha \in [10M,100M], K_1 \Pi(\alpha) \le \frac 1 2 
\end{equation}
where $\Pi$ is defined 
above Proposition A.2 in \cite{Gouere-Theret-17}.

Moreover, by \eqref{e:charme} and as $K_1 \le K$,
\begin{equation}\label{e:charme-2}
K_1^2 \int \cR^d(u) \1_{[10M,+\infty)}(\cR(u)) \d u \le  K^2 \int \cR^d(u) \1_{[10M,+\infty)}(\cR(u)) \d u \le \frac 1 4.
\end{equation}

The map $\epsilon$ is defined by (A.1) in \cite{Gouere-Theret-17}.
For all $\alpha>0$,
\begin{equation}\label{e:epsilon0}
\lambda \epsilon(\alpha) = \int_{[\alpha,+\infty)} r^d \nu^\cR(\d r) = \int \cR^d(u) \1_{[\alpha,+\infty)}(\cR(u)) \d u 
\end{equation}
where the last equality is a consequence of the definition of $\nu^\cR$.
Note that \eqref{e:chene-3} is (A.7) in \cite{Gouere-Theret-17} for $10M$ instead of $M$.
Moreover, (A.6) in \cite{Gouere-Theret-17} for $10M$ instead of $M$ is a consequence of \eqref{e:charme-2} and \eqref{e:epsilon0}.
Therefore, applying Lemma A.3 in \cite{Gouere-Theret-17}, we get 
\begin{equation}\label{e:Pi0}
\lim_{\alpha\to\infty} \Pi(\alpha)=0.
\end{equation}
By (A.2) in \cite{Gouere-Theret-17}, \eqref{e:Pi0}, \eqref{e:epsilon0} and \eqref{e:map-moment} we thus get \eqref{e:acacia}.
\end{proof}

\begin{proof}[Proof of Lemma \ref{l:subopen}]
Let $\Xi$ be a Poisson point process on $\R^d \times (0,+\infty)$ with intensity measure the Lebesgue measure.
As in Section \ref{s:couplings}, we build a point process $\hat\xi$ using $\Xi$ and $\hat\cR$ and introduce an appropriate random map $\hat r$ such that
\[
\hat\xi = \{(c,\hat r(c)), c \in \R^d \text{ such that } \hat r(c) >0\}.
\]
As in the proof of Theorem \ref{t:main}, we have shortened our notations.
In a similar way and with the same kind of notations, we define $\xi_n$ with $\cR_n$ for all $n \in \overline\N = \N \cup \{\infty\}$.
Thus, for all $n \in \overline\N$,
\[
\xi^n = \{(c,r_n(c)), c \in \R^d \text{ such that } r_n(c) >0\}.
\]
Moreover
\begin{equation}\label{e:pluspetit}
\forall n \in \overline\N, \forall c \in \R^d, r_n(c) \le \hat r(c).
\end{equation}
We denote by $\Sigma^n$ the Boolean model associated with $\xi^n$.
For all $n \in \overline\N$, define the event 
\[
C_n = \{\text{There exists a path in } \Sigma_n \text{ from } B(0,M) \text{ to } B(0,2M)^c\}.
\]
Let $K$ be the constant given by Lemma \ref{l:lemme-ouvert-GT17}.
By Items 2 and 3 of Lemma \ref{l:subopen}, there exists $M>0$ such that
\begin{equation}\label{e:it2}
K^2 \int \hat\cR^d(u) \1_{[M,+\infty)} (\hat\cR(u))\d u \le \frac 1 4
\end{equation}
and
\begin{equation}\label{e:it3}
K \P[C_\infty] \le \frac 1 4.
\end{equation}

Write
\[
\cC = \{c \in \R^d \text{ such that } \hat r(c) >0 \text{ and } B(c,\hat r(c)) \text{ touches } \overline{B(0,2M)}\}.
\]
For all $n \in \overline\N$, the event $C_n$ only depends on random balls that touch $\overline{B(0,2M)}$.
By \eqref{e:pluspetit}, it therefore only depends on the sequence $(r_n(c))_{c \in \cC}$.

By Item 4 of Lemma \ref{l:subopen}, there exists a full probability event $G_1$ on which, 
for all $c \in \R^d$, $r_n(c)$ converges to $r(c)$.
We refer to the proof of Claim \ref{claim:cvrn-} for a proof.
By the second item of Lemma \ref{l:subopen}, the expected value of the cardinality of $\cC$ is finite.
Therefore there exists a full probability event $G_2$ on which  $\cC$ is finite.
Finally there also exists a full probability event $G_3$ on which:
\begin{itemize}
\item For all distinct $(c,r)$ and $(c',r')$ in $\xi_\infty$, $B(c,r)$ and $B(c',r')$ are not tangent.
\item For all $(c,r)$ in $\xi_\infty$, $B(c,r)$ is neither tangent to $S(0,M)$ nor to $S(0,2M)$.
\end{itemize}
We work on the full probability event $G=G_1 \cap G_2 \cap G_3$.
On $G$, the finite sequence $(r_n(c))_{c \in \cC}$ converges to $(r_\infty(c))_{c \in \cC}$ and therefore 
$\1_{C_n}$ converges to $\1_{C_\infty}$.
Therefore $\P[C_n]$ converges to $\P[C_\infty]$.
From \eqref{e:it3} we deduce the existence of $n_0$ such that, for all $n \ge n_0$, $\cR_n$ satisfies \eqref{e:chene}.
By \eqref{e:it2} and Item 1 of Lemma \ref{l:subopen}, for all $n \ge n_0$, $\cR_n$ satisfies \eqref{e:charme}.
Finally, by Items 1 and 2 of Lemma \ref{l:subopen}, for all $n \ge n_0$, $\cR_n$ satisfies \eqref{e:map-moment}.
By Lemma \ref{l:lemme-ouvert-GT17} we deduce that, for all $n \ge n_0$, $\cR_n$ is strongly subcritical.
\end{proof}

\section{BK inequality}
\label{s:bk}

\subsection{Framework and result}

\paragraph{Disjoint occurrence.}
Let $p \ge 1$.
Let $\cS$ denote here the set of finite subsets of $\R^p$.
Let $\cF$ denote the usual $\sigma$-field for finite point processes on $\R^p$.
This is the $\sigma$-field on $\cS$ generated by the application $S \mapsto \card(S \cap A)$ for Borel sets $A \subset \R^p$.
We say that $F \in \cF$ is increasing if, for any configuration $S$ in $F$, any configuration $S' \supset S$ also belongs to $F$.
If $F_1, \dots, F_m$ are $m$ increasing elements of $\cF$, we define the disjoint occurrence of $F_1,\dots,F_m$ by
\[
F_1 \circ \cdots \circ F_m = \{S \in \cS : \text{ there exists pairwise disjoint } S_1, \dots, S_m \subset S \textrm{ s.t. } S_1 \in F_1, \dots, S_m \in F_m\}.
\]
\paragraph{Measurability issues.}
There exists a sequence $(X_n)_n$ of measurable maps from $(\cS,\cF)$ to $\R^p$ and a measurable map $N$ from $(\cS,\cF)$ to $\N$ such that,
for all $S \in \cS$,
\[
S = \{X_1(S), \dots, X_{N(S)}(S)\} \text{ and } X_1(S), \dots, X_{N(S)}(S) \text{ are distinct.}
\]
Therefore, $F_1 \circ \cdots \circ F_m$ is also the set of all configurations $S \in \cS$ such that there exists pairwise disjoint 
$I_1, \dots, I_m \subset \{1,\dots,N(S)\}$ for which
\[ 
\{X_i(S), i \in I_1\} \in F_1, \dots, \{X_i(S), i \in I_m\} \in F_m.
\]
This rewriting enables to check that $F_1 \circ \cdots \circ F_m$ belongs to the $\sigma$-field $\cF$.

\paragraph{Main result.} 
We now state our main result. This is a version of BK inequality for Poisson point processes which is similar to the version given in the discrete setting by Kesten in its Saint-Flour lecture notes \cite{Kesten-saint-flour}.
As in all the proofs of BK inequality we are aware of, our proof is fundamentally relying on some splitting techniques.
However, contrary to other proofs in the continuum setting ,we do not perform any discretization step and the key property is captured in a purely combinatorial lemma.
\begin{lemma} \label{l:bk}
Let $p \ge 1$ and $m \ge 2$.
Let $\varpi$ be a finite measure on $\R^p$ with no atoms.
Let $\chi_1, \dots, \chi_m$ be independent Poisson point processes on $\R^p$ with intensity measure $\varpi$.
Let $(F_i^t)_{i \in \{1,\dots,m\}, t \in T}$ be a family of increasing events on $\cF$.
We assume that $T$ is finite or countable.
Then
\[
\P\left[\bigcup_{t \in T}\{ \chi_1 \in F_1^t \circ \cdots \circ F_m^t\}\right] 
\le 
\P\left[\bigcup_{t \in T}\{ \chi_1 \in F_1^t,  \dots, \chi_m \in  F_m^t\}\right].
\]
\end{lemma}

For a proof of standard BK inequality for Poisson point processes, we refer to the paper by van den Berg \cite{vdB96} (see also the book by Meester and Roy \cite{Meester-Roy-livre}) and references therein. 
Earlier proofs of BK inequality for Poisson point processes (see for instance Bezuidenhout and Grimmett \cite{BezuidenhoutGrimmett91}) were built in three steps:
partition the space in small cells, 
apply the classical BK inequality to a collection of independent variables indexed by these cells
and then take the limit of the desired probabilities as the size of the cells goes to 0.
This approach, through the use of a limit in the third step, requires some extra regularity assumptions on the events we can consider. 
Van den Berg offered in \cite{vdB96} a new approach to the problem: 
he still starts by partitioning the space in cells, but then instead of using the discrete version of BK inequality he adapts directly its proof.
More precisely, he uses the splitting method, a well known method to prove classical BK inequality in which one replaces one after each other the variable associated to each edge or each cell by an independent copy.
Van den Berg has to deal with an error term due to the discretization (more precisely to the existence of at least two points in a cell) but this term goes to $0$ with the size of the cells.

\subsection{A combinatorial lemma}
\label{s:combinatoire}

Let $n \ge 1$ and  $m \ge 2$.
Let
\[
\cQ = \cQ(n,m) = \big\{(Q_1,\dots,Q_m) \text{ where } Q_1, \dots Q_m \text{ are pairwise disjoint subsets of } \{1,\dots,n\}\big\}.
\]
For any subset $\cR \subset \cQ$ we define
\[
\cA_n(\cR) = \bigcup_{(R_1,\dots,R_m) \in \cR} \{(c_1,\dots,c_n) \in \{1,\dots,m\}^n : 
\text{ for all } a \in \{1,\dots,m\} \text{ and all } i \in R_a, c_i=a\}.
\]
One can think about the family $(c_1,\dots,c_n)$ as a coloring of the sites $\{1,\dots,n\}$ with colors in $\{1,\dots,m\}$
and about the set $\cR$ as a set of rules.
A coloring is $(R_1,\dots,R_m)$-admissible if for all color $a \in \{1,\dots,m\}$, all sites of $R_a$ are colored with color $a$.
The set $\cA_n(\cR)$ is then the set of coloring which are $R$-admissible for some $R \in \cR$.
We also define
\[
\cA_0(\cR) = \bigcup_{(R_1,\dots,R_m) \in \cR} \{(c_1,\dots,c_n) \in \{1,\dots,m\}^n : 
\text{ for all } a \in \{1,\dots,m\} \text{ and all } i \in R_a, c_i=1\}.
\]
The only difference is that we require all sites of all $R_a$ to be colored with color $1$.

\begin{lemma} \label{l:combinatoire} Let $n \ge 1$ and  $m \ge 2$.
Let $\cR \subset \cQ(n,m)$. Then
\[
\card(\cA_0(\cR)) \le \card(\cA_n(\cR)).
\]
\end{lemma}
\begin{proof} For all $k \in \{0,\dots,n\}$ let
\[
\cA_k(\cR) = \bigcup_{R \in \cR} \{c \in \{1,\dots,m\}^n : c \text{ is } (k,R) \text{ admissible}\}
\]
where "$c \text{ is } (k,R) \text{ admissible}$" means the following:
\[
\text{for all }a \in \{1,\dots,m\}\text{ and for all }i \in R_a,\text{ if } i \le k,\text{ then } c_i=a \text{ else } c_i = 1.
\]
This definition is compatible with the previous definitions for $k=0$ and $k=n$.
The lemma will follow immediately from the fact that, for all $k \in \{1,\dots,n\}$,
\begin{equation}\label{e:combinfort}
\card(\cA_{k-1}(\cR)) \le \card(\cA_k(\cR)).
\end{equation}
Let $k \in \{1,\dots,n\}$.
For all $\hat c=(\hat c_1,\dots,\hat c_{k-1},\hat c_{k+1},\dots,\hat c_n) \in \{1,\dots,m\}^{n-1}$, define
\[
\cB_{k-1}(\hat c, \cR) = \{c_k \in \{1,\dots,m\} : (\hat c_1,\dots,\hat c_{k-1},c_k, \hat c_{k+1},\dots,\hat c_n) \in \cA_{k-1}(\cR)\}
\]
and
\[
\cB_k(\hat c, \cR) = \{c_k \in \{1,\dots,m\} : (\hat c_1,\dots,\hat c_{k-1},c_k, \hat c_{k+1},\dots,\hat c_n) \in \cA_k(\cR)\}.
\]
As $\card(\cA_{k-1}(\cR)) = \sum_{\hat c} \card (\cB_{k-1}(\hat c, \cR) )$ and $\card(\cA_k(\cR)) = \sum_{\hat c} \card (\cB_k(\hat c, \cR) )$,
\eqref{e:combinfort} follows from the stronger result:
\begin{equation}\label{e:combinfortfort}
\forall \hat c \in\{1,\dots,m\}^{n-1}, \card(\cB_{k-1}(\hat c,\cR)) \le \card(\cB_k(\hat c,\cR)).
\end{equation}
Let $\hat c \in\{1,\dots,m\}^{n-1}$. There are two cases:
\begin{enumerate}
\item Case $\cB_{k-1}(\hat c,\cR)=\emptyset$. In this case, \eqref{e:combinfortfort} is straightforward. 
This is not needed for the proof, but in this case we actually also have $\cB_k(\hat c,\cR)=\emptyset$.
\item Case $\cB_{k-1}(\hat c,\cR) \neq \emptyset$.
In other words, there exists $c_k \in \{1,\dots,m\}$ and $R \in \cR$ 
such that $(\hat c_1,\dots,\hat c_{k-1},c_k,\hat c_{k+1},\dots,\hat c_n)$ is $(k-1,R)$ admissible.
We now subdivide in two subcases.
\begin{enumerate}
\item There exists $c_k \in \{1,\dots,m\}$ and $R \in \cR$ 
such that $(\hat c_1,\dots,\hat c_{k-1},c_k,\hat c_{k+1},\dots,\hat c_n)$ is $(k-1,R)$ admissible
and such that $k \not\in \bigcup_a R_a$.
Fix such a $R$.
As a first consequence, being $(k-1,R)$ admissible does not depend on the $k$th coordinate.
Therefore $\cB_{k-1}(\hat c,\cR)=\{1,\dots,m\}$.
As a second consequence, being $(k-1,R)$ admissible is equivalent to being $(k,R)$ admissible.
Therefore $\cB_k(\hat c,\cR)=\{1,\dots,m\}$.
In particular, \eqref{e:combinfortfort} holds.
\item For all $c_k \in \{1,\dots,m\}$ and $R \in \cR$ 
such that $(\hat c_1,\dots,\hat c_{k-1},c_k,\hat c_{k+1},\dots,\hat c_n)$ is $(k-1,R)$ admissible, $k$ belongs to $\bigcup_a R_a$.
Therefore $\cB_{k-1}(\hat c,\cR)=\{1\}$.
Let $R \in \cR$ be such that $(\hat c_1,\dots,\hat c_{k-1},1,\hat c_{k+1},\dots,\hat c_n)$ is $(k-1,R)$ admissible.
Let $a$ be such that $k \in R_a$.
Then $(\hat c_1,\dots,\hat c_{k-1},a,\hat c_{k+1},\dots,\hat c_n)$ is $(k,R)$ admissible 
--this uses the fact that $R_1, \dots, R_m$ are pairwise disjoint -- and therefore $\{a\} \subset \cB_k(\hat c,\cR)$.
As a consequence, \eqref{e:combinfortfort} holds.
Note that this is the only case where $\cB_{k-1}(\hat c,\cR)$ and $\cB_k(\hat c,\cR)$ can be different.
\end{enumerate}
\end{enumerate}
We proved \eqref{e:combinfortfort}. This yields \eqref{e:combinfort} and then the lemma.
\end{proof}

\subsection{Proof of Lemma \ref{l:bk}}

Write $\lambda = \varpi(\R^p)$.
If $\lambda=0$, then $\chi_1, \dots, \chi_m$ are empty and the result is trivial.
Henceforth, we assume $\lambda>0$ and write $\hat\varpi=\lambda^{-1}\varpi$.
Let $N$ be a Poisson random variable with parameter $m\lambda$.
Let $X=(X_i)_i$ be a family of i.i.d.r.v.\ on $\R^p$ with distribution $\hat\varpi$.
Let $(U_i)_i$ be a family of i.i.d.r.v.\ with uniform distribution on $\{1,\dots,m\}$.
The sets
\[
\chi_a = \{X_i : i \in \{1,\dots,N\} : U_i=a\}, a \in \{1,\dots, m\}
\]
are independent Poisson point processes with intensity measure $\varpi$.
We can thus work with these realizations of the point processes.

Recall the notation of Section \ref{s:combinatoire} and set
\[
\cR(N,X) = \{ (R_1,\dots,R_m) \in \cQ(N,m) : \exists t \in T \textrm{ s.t. } \{X_i, i \in R_1\} \in F^t_1, \dots, \{X_i, i \in R_m\} \in F^t_m\}.
\]
With this notation, we can write (using the fact that the $X_i$ are almost surely distinct)
\begin{align}
\label{e:ajoutfinal1}
 & \P\left[ \chi_1 \in \bigcup_{t \in T} F_t^1 \circ \cdots \circ F^t_m\right] \nonumber \\
 & = \P\left[  \exists (R_1,\dots,R_m) \in \cR(N,X) 
 : \forall a \in \{1,\dots,m\} \; \forall i \in R_a, \; U_i=1\right] \nonumber \\
 & = \E\Big[ \P\big(   \exists (R_1,\dots,R_m) \in \cR(N,X) 
 : \forall a \in \{1,\dots,m\} \; \forall i \in R_a, \; U_i=1\Big|  N,X\big)\Big]
\end{align}
and
\begin{align}
\label{e:ajoutfinal2}
& \P\left[ \bigcup_{t \in T} \{\chi_1 \in F_1^t, \dots, \chi_m \in F_m^t\}\right] \nonumber \\
& = \P\left[  \exists (R_1,\dots,R_m) \in \cR(N,X) 
 : \forall a \in \{1,\dots,m\} \; \forall i \in R_a, \; U_i=a\right] \nonumber \\
 & = \E\Big[ \P\big(   \exists (R_1,\dots,R_m) \in \cR(N,X) 
 : \forall a \in \{1,\dots,m\} \; \forall i \in R_a, \; U_i=a\Big|  N,X\big)\Big].
 \end{align}
The variables $U_i$ have uniform distribution on $\{ 1,\dots , m\}$, thus the conditional probabilities appearing in \eqref{e:ajoutfinal1} and \eqref{e:ajoutfinal2} are proportional to the cardinal of the corresponding sets of admissible colorings. The result then follows from Lemma \ref{l:combinatoire}. \qed

%\bibliographystyle{plain}
%\bibliography{biblio.bib}

\begin{thebibliography}{10}

\bibitem{AuffingerDamronHanson}
Antonio Auffinger, Michael Damron, and Jack Hanson.
\newblock 50 years of first passage percolation.
\newblock In {\em University Lecture Series}, volume~68. American Mathematical
  Society, 2017.

\bibitem{bates2020empirical}
Erik Bates.
\newblock Empirical distributions, geodesic lengths, and a variational formula
  in first-passage percolation, 2020.

\bibitem{vdB96}
J.~Van~Den Berg.
\newblock A note on disjoint-occurrence inequalities for marked poisson point
  processes.
\newblock {\em Journal of Applied Probability}, 33(2):420--426, 1996.

\bibitem{BezuidenhoutGrimmett91}
Carol Bezuidenhout and Geoffrey Grimmett.
\newblock Exponential decay for subcritical contact and percolation processes.
\newblock {\em The Annals of Probability}, 19(3):984--1009, 1991.

\bibitem{DamronTang}
Michael Damron and Pengfei Tang.
\newblock Superlinearity of geodesic length in 2d critical first-passage
  percolation.
\newblock In Vladas Sidoravicius, editor, {\em Sojourns in Probability Theory
  and Statistical Physics - II}, pages 101--122, Singapore, 2019. Springer
  Singapore.

\bibitem{GM-deijfen}
Jean-Baptiste Gou{\'e}r{\'e} and R{\'e}gine Marchand.
\newblock Continuous first-passage percolation and continuous greedy paths
  model: linear growth.
\newblock {\em Ann. Appl. Probab.}, 18(6):2300--2319, 2008.

\bibitem{Gouere-Theret-17}
Jean-Baptiste Gouéré and Marie Théret.
\newblock Positivity of the time constant in a continuous model of first
  passage percolation.
\newblock {\em Electron. J. Probab.}, 22:21 pp., 2017.

\bibitem{Grimmett-percolation}
Geoffrey Grimmett.
\newblock {\em Percolation}, volume 321 of {\em Grundlehren der Mathematischen
  Wissenschaften [Fundamental Principles of Mathematical Sciences]}.
\newblock Springer-Verlag, Berlin, second edition, 1999.

\bibitem{Kesten-saint-flour}
Harry Kesten.
\newblock Aspects of first passage percolation.
\newblock In {\em \'Ecole d'\'et\'e de probabilit\'es de Saint-Flour,
  XIV---1984}, volume 1180 of {\em Lecture Notes in Math.}, pages 125--264.
  Springer, Berlin, 1986.

\bibitem{Last-Penrose-livre}
G\"{u}nter Last and Mathew Penrose.
\newblock {\em Lectures on the {P}oisson process}, volume~7 of {\em Institute
  of Mathematical Statistics Textbooks}.
\newblock Cambridge University Press, Cambridge, 2018.

\bibitem{Martin-greedy}
James~B. Martin.
\newblock Linear growth for greedy lattice animals.
\newblock {\em Stochastic Process. Appl.}, 98(1):43--66, 2002.

\bibitem{Meester-Roy-livre}
Ronald Meester and Rahul Roy.
\newblock {\em Continuum percolation}, volume 119 of {\em Cambridge Tracts in
  Mathematics}.
\newblock Cambridge University Press, Cambridge, 1996.

\bibitem{ZHANGsurcritique}
Yu~Zhang.
\newblock Supercritical behaviors in first-passage percolation.
\newblock {\em Stochastic Processes and their Applications}, 59(2):251 -- 266,
  1995.

\bibitem{ZhangZhang}
Yu~Zhang and Yi~Ci Zhang.
\newblock A limit theorem for {$N_{0n}/n$} in first-passage percolation.
\newblock {\em Ann. Probab.}, 12(4):1068--1076, 1984.

\end{thebibliography}

\def\cprime{$'$} \def\cprime{$'$}

\end{document}